\documentclass[11pt,letterpaper]{article}
\usepackage{setspace}
\usepackage[usenames,dvipsnames]{color}

\usepackage{fullpage,amsfonts,amsmath,amsthm,amssymb,mathrsfs,graphicx,epstopdf}
\usepackage[top=1in, bottom=1in, left=1in, right=1in]{geometry}
\usepackage{amsfonts,amsmath,latexsym,hyperref,graphicx,amssymb,amsthm,url,lineno,float,tikz}
\usepackage[]{algorithm2e}
\usepackage[abbrev,msc-links]{amsrefs}  
\usetikzlibrary{arrows}
\usepackage{rotating}
\usetikzlibrary{calc}
\usetikzlibrary{positioning}
\usetikzlibrary{patterns}
\usetikzlibrary{shapes}
\usepackage{subcaption}
\usepackage{tikz}
\usepackage{bbm}
\usetikzlibrary{arrows.meta}
\usepackage[normalem]{ulem}

\newcommand{\lab}[1]{\label{#1}}           

\newcommand{\remove}[1]{}
\newcommand\eqn[1]{(\ref{#1})}

\newcommand{\be}{\begin{equation}}
\newcommand{\bel}[1]{\begin{equation}\lab{#1}\ }
\newcommand{\ee}{\end{equation}}
\newcommand{\bea}{\begin{eqnarray}}
\newcommand{\eea}{\end{eqnarray}}
\newcommand{\bean}{\begin{eqnarray*}}
	\newcommand{\eean}{\end{eqnarray*}}

\newtheorem{thm}{Theorem}

\newtheorem{lemma}[thm]{Lemma}
\newtheorem{definition}[thm]{Definition}
\newtheorem{claim}[thm]{Claim}

\newtheorem{remark}[thm]{Remark}

\def\qed{~~\vrule height8pt width4pt depth0pt}
\def\ss{{\smallskip}}

\def\totmultcut{t_0}
\def\totmult{t}

\def\UB{{\overline f}}
\def\LB{{\underline b}}

\def\MatrixGen{{\tt MATRIXGEN}}
\def\GraphGen{{\tt MULTIGRAPHGEN}}
\def\SimpleGen{{\tt INC-GEN}}
\def\SimpleBiGen{{\tt INC-BIPARTITE}}
\def\Gen{{\tt Gen}}
\def\Brute{{\tt Brute}}
\def\sbrute{{\tt SubBrute}}

\def\A{{\mathcal A}}

\def\D{{\mathcal D}}

\def\H{{\mathcal H}}

\def\M{{\mathcal M}}

\def\pr{{\mathbb P}}

\def\bfa{{\bf a}}

\def\bfd{{\bf d}}
\def\bfe{{\bf e}}

\def\bfg{{\bf g}}
\def\bfh{{\bf h}}

\def\bfm{{\bf m}}

\def\bfp{{\bf p}}

\def\bfs{{\bf s}}
\def\bft{{\bf t}}

\def\inS{G^{-}(S)}
\def\outS{G^{+}(S)}

\def\bbN{{\mathbb N}}

\def\eps{\epsilon}

\date{}
\title{Linear-time uniform generation of random sparse contingency tables with specified marginals}
\author{Andrii Arman\and Pu Gao\and Nicholas Wormald}
\author{
	Andrii Arman\\
	School of Mathematics\\
	Monash University\\
	andrii.arman@monash.edu
	\and
	Pu Gao\thanks{Research supported by ARC DP160100835 and NSERC.}\\
	Department of Combinatorics and Optimization\\
	University of Waterloo\\
	pu.gao@uwaterloo.ca
	\and
	Nicholas Wormald\thanks{Research supported by  ARC DP160100835.}\\
	School of Mathematics\\
	Monash University\\
	nick.wormald@monash.edu }
\begin{document}
	\maketitle
	
	\begin{abstract}
		
		We give an algorithm which generates a uniformly random contingency table with specified marginals, i.e.\ a matrix with non-negative integer values and specified row and column sums.   Such algorithms are useful in statistics and combinatorics.   When $\Delta^4< M/5$,    where $\Delta$ is the maximum of the row and column sums and $M$ is the sum of all entries of the matrix, our algorithm runs in time linear in $M$ in expectation. Most previously published algorithms for this problem are approximate samplers based on Markov chain Monte Carlo, whose provable bounds on the mixing time are typically polynomials with rather large degrees. 
	\end{abstract}
	\thispagestyle{empty}
	\newpage
	\setcounter{page}{1}
	
	\section{Introduction and main results}

	Let  $\bfs=(s_1,s_2,\ldots, s_m)$ and $\bft=(t_1,t_2,\ldots, t_n)$  be two vectors of positive integers such that $\sum_{i=1}^m s_i =\sum_{j=1}^n t_j=M$. A contingency table with marginals $(\bfs, \bft)$ is an $m\times n$ matrix with nonnegative integer entries such that the sum of entries in the $i$-th row is $s_i$ and the sum of entries in the $j$-th column is $t_j$, for every $i\in [m]$ and $j\in [n]$.   In this paper, we provide an algorithm,  \MatrixGen, to generate  a uniformly random contingency table with specified marginals.   \MatrixGen\  has two advantages over all previous  algorithms for this problem: firstly, it is exact, so there is no approximation error, and secondly  it runs much faster than the previous algorithms:  the expected runtime is linear in $M$ provided that the maximum of the $s_i$ and $t_j$ is at most $(M/5)^{1/4}$.  While most previously research on this problem uses  Markov Chain Monte Carlo  (MCMC), which we describe below, we base our work instead on the switching method. This  has been used in the past to generate graphs with given degrees uniformly at random, and requires significant modification to apply to non-binary matrices. The most important ingredient for achieving linear time is a technique that we recently developed in~\cite{arman19}.

	The problem of how to uniformly generate members of a finite set of objects has a long history. Early works include those by Wilf~\cite{wilf77,wilf81} on uniform generation of trees, and other instances of combinatorial objects with recurrent structure.  Jerrum, Valiant and Vazirani~\cite{jerrum86} introduced a unified notation for random generation and for complexity classes of generation problems. In particular they observed that, roughly speaking, uniform generation is no harder than counting, whereas  
	approximate generation is equally  as  hard as approximate counting. 
	Several generic techniques are commonly used for random (often approximate) generation. The most commonly applied method is the so-called Markov Chain Monte Carlo (MCMC) method. It defines a Markov chain on the set $\Omega$ of objects which we aim to generate uniformly.   The Markov chain is designed so that it is ergodic and its stationary distribution is uniform. Then we can run the chain sufficiently long (measured by the mixing time~\cite{levin17}) and then output.  	The MCMC method is efficient if the mixing time of the Markov chain is small. The chain is called  ``rapidly mixing'' if its mixing time  is bounded by some  polynomial in the size of the object being generated.  However, for many Markov chains designed for combinatorial generation problems, the degree of this polynomial is   large, too large for practical use of the MCMC method. 
	Another technique, called   rejection sampling, works when there is an efficient algorithm for generation of objects in a larger space $\Omega^*$ containing $\Omega$, where all objects in $\Omega$ appear with equal probability. The rejection scheme then rejects each generated object outside of $\Omega$ until finding an object in $\Omega$, and outputs it. This scheme is consequently an exactly uniform sampler. However. if $|\Omega|/|\Omega^*|$ is too small then the rejection scheme is not efficient. In some cases, the switching method~\cite{mckay90} can be used to boost  the efficiency by progressively transforming   an object in $\Omega^*$, using repeated steps that maintain a certain uniformity property, until reaching some object in $\Omega$. These separate steps also incorporate   rejection schemes.

	Contingency tables are extensively used in social sciences, statistics and medicine research, where categorical data is analysed~\cite{everitt92,fagerland17}.  Exact counting of contingency tables with specified marginals is known to be \#P-complete, even when $m$ or $n$ is equal to 2  (see Dyer, Kannan and Mount~\cite{dyer97}), and can be done only in special cases  (see e.g.\ Barvinok~\cite{barvinok94}). There are no known polynomial time algorithms for approximately uniform sampling or approximate counting of contingency tables when the marginals are arbitrary. When the number $m$ of rows  is constant, Cryan and Dyer~\cite{cryan03}, and Dyer~\cite{dyer03} gave algorithms for approximate counting and approximately uniform sampling where the runtime is polynomial in $n$ and $\log M$. The former is based on volume estimation and the latter uses dynamic programing. The first Markov chain for (approximately) uniformly sampling contingency tables, which had already been actively used by statisticians at that time, was analysed by Diaconis and Gangolli~\cite{diaconis95} and by Diaconis and Saloff-Coste~\cite{diaconis95walk}. Roughly speaking, the chain chooses a $2\times 2$ submatrix, and alters the entries in the submatrix by at most 1 subject to the marginal constraint.  For convenience we call this the Diaconis-Gangolli chain. Diaconis and Gangolli~\cite{diaconis95} proved that this Markov chain is ergodic and converges to the uniform distribution, without bounding the mixing time. In the case where both $n$ and $m$ are constant, Diaconis and  Saloff-Coste~\cite{diaconis95walk} proved that the mixing time is at most quadratic in $M$. Hernek~\cite{hernek98} studied the Diaconis-Gangolli  chain in the case $m=2$ and proved that the mixing time is bounded by a polynomial in $n$ and $\log M$. A different but related chain was studied in~\cite{chung96} and the mixing time is bounded by a polynomial in $n$, $m$, and $M$, if  the row and column sums are sufficiently large compared with $m$ and $n$. Using a different approach, Dyer, Kannan and Mount~\cite{dyer97} obtained the first fully  polynomial algorithm (polynomial in $n$, $m$ and $\log M$) for approximate counting and sampling of the contingency tables, provided that the row sums are at least $n^2m$ and the column sums are at least $m^2n$. Their condition on the row and column sums was slightly relaxed  by Morris~\cite{morris02}. Dyer and Greenhill~\cite{dyer00} considered a different Markov chain, which  randomly chooses a $2\times 2$ submatrix, and then replaces the submatrix by a uniformly random $2\times 2$ matrix subject to the marginal restrictions.  
	They studied the case where $m=2$ and proved that the mixing time is polynomial in $n$ and $\log M$. Later this result was extended to the case of arbitrary fixed $m$ by Cryan, Dyer, Goldberg, Jerrum and Martin~\cite{cryan06}, with mixing time bounded by a polynomial in $n$ and $\log M$. In a recent  PhD thesis,  Dittmer~\cite{dittmer19} reported some new results on approximate counting and sampling of random sparse contingency tables. In his work a new Markov chain (with nontrivial transitions) is introduced and the chain is rapidly mixing if $n$ and $m$ are of the same order, and the row and column sums are either of order up to $n^{1/4-\eps}$, or of equal order up to $n^{1-\eps}$. The runtime of each transition in the chain is polynomial and thus his algorithm yields a polynomial-time approximate sampler.
	
	Besides the aforementioned approximate samplers, Chen, Diaconis, Holmes and Liu~\cite{chen05} used  sequential importance sampling (SIS) to sample and count contingency tables. Their algorithm runs in polynomial time empiricially but they do not provide a theoretical  bound on the number of samples required to guarantee  any given error of approximation. Blanchet~\cite{blanchet09} proved that $O(M^3\eps^{-2}\delta^{-1})$ samples are sufficient for  the Chen-Diaconis-Holmes-Liu SIS method to guarantee an $\eps$-approximation with probability $1-\delta$, under the condition that all row sums are bounded, all column sums are $o(M^{1/2})$, and the sum of all column sums squared is $O(M)$. On the other hand, negative examples were provided by Bez{\'a}kov{\'a}, Sinclair, {\v{S}}tefankovi{\v{c}} and Vigoda~\cite{bezakova12} where an exponential number of samples are necessary.
	
	The problem simplifies if the contingency tables are restricted to binary (0/1) entries, and they are then equivalent to bipartite graphs, with the marginals specifying the degrees of the vertices. This has received considerable attention~\cite{tinhofer79,jerrum90,cooper07,greenhill14,rao96,verhelst08,blitzstein11,bollobas80,mckay90,gao17,gao18,arman19,steger99,kim03,bayati10,zhao13}.

	Our new algorithm, \MatrixGen, will be defined in Section~\ref{s:overview}. Throughout the paper we always assume	\[
	\sum_{i=1}^m s_i =\sum_{j=1}^n t_j.
	\]
	This condition is necessary as otherwise there is no contingency table with the prescribed marginals.
	Let
	\[
	M=\sum_{i=1}^m s_i, \quad \Delta=\max_{1\le i\le m, 1\le j\le n} \{s_i, t_j\}.
	\]
	To avoid another triviality we may also assume that $\bfs$ and $\bft$ have only positive components, since otherwise we may consider $(\bfs',\bft')$ obtained by deleting all 0 components from $\bfs$ and $\bft$. Thus we have the handy relation $M= \Omega(m+n)$. We say that $(\bfs, \bft)$ is \emph{bi-graphical} if there exists a simple bipartite graph whose two parts have degree sequences   $\bfs$ and $\bft$ respectively.  
	
	The main result is as follows.
	\begin{thm}\lab{thm:matrix}
		If $(\bfs, \bft)$ is bi-graphical then algorithm \MatrixGen\  generates a uniformly random $m\times n$  contingency table with marginals  $(\bfs,\bft)$.   \MatrixGen\ has expected time complexity $O(M)$ when  $5\Delta^4<M$, and has deterministic space complexity    $O(mn\log(\Delta+1))$ in all cases.
	\end{thm}
	\begin{remark}
	\begin{itemize}
		\item[(a)] The condition $5\Delta^4<M$ implies that $(\bfs,\bft)$ is bi-graphical by the Gale-Ryser theorem.
		
		\item[(b)] Since time complexity  is determined by parts of  the algorithm that only involve  numbers of size  $O(M)$, in evaluating  this complexity we assume arithmetic operations require  only $O(1)$ time. However, the space complexity involves storing much larger  numbers, so we evaluate the space required according  to the number of bits.
		
		\item[(c)] We are not aware of any fully polynomial approximate samplers in general. In \cites{diaconis95walk, hernek98, dyer00} $m$ is required to be fixed, whereas in~\cite{dittmer19} it is assumed that $m=\Theta(n)$. Our sampler covers some other ranges of $m$. It is an exact uniform sampler and runs much faster than those in~\cite{diaconis95walk, hernek98, dyer00,dittmer19}. By terminating the algorithm prematurely in case of some rare events, we can obtain a practical approximate sampler running in  linear expected time and in time $O(M^2\log M)$  always. The rare events that require termination are the occurrence of more than $M\log M$ restarts, or that the algorithm enters a very slow computation: procedure \Brute{} described in  Section~\ref{sec:Brute}. The probability of each of these rare events is $O(e^{-\Omega(M\log M)})$ so the output distribution of the approximate sampler differs from uniform by $O(e^{-\Omega(M\log M)})$ in total variation distance.
		\end{itemize}
	\end{remark}

	Contingency tables with specified marginals are in one-to-one correspondence to bipartite multigraphs with prescribed vertex degrees. We will  use the language of bipartite multigraphs instead of contingency tables as it is easier to describe the algorithm using graph theoretic terminology. Our approach is completely different from~\cite{cryan03,dyer03} and all the MCMC-based algorithms. Instead, we proceed along  the lines of~\cite{mckay90,gao17,gao18,arman19}, i.e.\ we design an exactly uniform sampler for  bipartite multigraphs with given degrees, using a switching method. The adjustments required for generating   multigraphs rather than graphs are far from straightforward, and we explain   why below.
	
	We first give a broad description of the switching method used in algorithms generating random graphs with given degrees. Initially,  a random multigraph is generated using the configuration model~\cite{bollobas80} introduced by Bollob\'{a}s:  represent each vertex as a bin containing a set of points whose number equals the degree of that vertex, then take a uniformly random perfect matching of the set of points to determine the edges. Given the multigraph, a sequence of switching operations are applied, each of which alters a few edges  but not the degrees of the vertices, such that eventually a simple graph is obtained. With a carefully designed rejection scheme, the output is uniformly random. For the bipartite case, the perfect matching is restricted so that only points in bins on different sides of the graph are matched.

	Generating random multigraphs has several difficulties not encountered in the generation of simple graphs. The configuration model generates a random multigraph. However, it is not distributed uniformly. The probability that a given multigraph occurs depends on how many multiple edges and loops of each multiplicity it contains. For instance, a multigraph that contains one multiple edge of multiplicity $m$ and no other multiple edges is $1/m!$ as likely to appear as any simple graph. A typical uniformly random multigraph would contain a large number of multiple edges if the degrees of the vertices are some power of $n$. Our algorithm starts from a uniformly random bipartite simple graph (by calling an existing linear-time algorithm), and then adds multiple edges using switchings. There are three main challenges here. 
	
	First we need to decide when this switching procedure  should stop adding multiple edges. There was no corresponding issue  for generation of simple graphs, since there the multiple edges are removed and the algorithm naturally stops when none remains. 
	
	The second challenge is in the design of a scheme for addition of multiple edges of high multiplicity. This was not a problem for generating simple graphs as with high probability the initial multigraph contains only multiple edges of low multiplicity. For instance, if the maximum degree is $o(n^{1/2})$ for a regular degree sequence on $n$ vertices,  then with high probability  there can only be simple loops, double edges and triple edges in the initial multigraph. Thus, high multiplicity edges were dealt with by simply rejecting the initial multigraph if it contains any.  However,  to  generate  random  multigraphs with exactly the uniform distribution, an algorithm must necessarily be capable of outputting every possible multigraph, including those with edges of very high multiplicity.  
	
	The third challenge is for the design of the rejection scheme. For generation of simple graphs, the rejection probabilities in each switching step are determined by computing the exact number of ways to perform a switching, or perform an inverse switching. This computation can be done efficiently --- easily in polynomial time --- if the switching operation only involves a small number of edges. This is the case for generating simple graphs. For multigraphs on the other hand,  occasionally  multiple edges with high multiplicity, possibly a power of $n$,   must be added. Following the ``standard'' procedure whereby the entire multiple edge is dealt with in one switching (which was the breakthrough in~\cite{mckay90} enabling super-logarithmic degrees to be treated) then leads to a super-polynomial time requirement for computing the exact number of ways a switching can be performed. In order to overcome this obstacle, we use a recently-developed rejection scheme~\cite{arman19} (by the same authors of this paper), which is different from~\cite{mckay90,gao17,gao18}, and can be implemented with small time cost. 
	In fact, it was contemplation of this very obstacle for multigraphs that evolved into the main new idea in~\cite{arman19}.

	With minor modifications, we also obtain a linear-time algorithm, \GraphGen, for generating random loopless multigraphs with a prescribed degree sequence. Roughly speaking, this uses similar switchings, but without being restricted by a vertex bipartition. The description of \GraphGen\ is given in Section~\ref{sec:multigraph}. We say that $\bfd$ is \emph{graphical} if there exists a simple  graph with degree sequence $\bfd$.
	\begin{thm}\lab{thm:multigraph}
		Assume $\bfd\in \bbN^n$ be such that $M=\sum d_i$ is even. Let $\Delta$ denote the maximum of the components in $\bfd$. If $\bf{d}$ is graphical, then algorithm \GraphGen\ uniformly generates a random loopless multigraph with degree sequence $\bfd$.  \GraphGen\ has expected time complexity $O(M)$ when  $5\Delta^4<M$, and has deterministic space complexity    $O(mn\log(\Delta+1))$ in all cases.
	\end{thm}
	
	Uniform generation of multigraphs permitting loops requires more work and we will address that problem in a subsequent paper.

	\section{Overview}
	 \lab{s:overview}

	Let $X$ and $Y$ be two sets of vertices with $|X|=m$ and $|Y|=n$. 
	Then  $(\bfs,\bft)$ is a {\em bipartite degree sequence} for $(X,Y)$ if $\bfs$ is $m$-dimensional,  $\bft$ is $n$-dimensional, and $\sum_{i=1}^m s_i=\sum_{j=1}^n t_j$.
	Recall that $\Delta$ denotes the maximum of all components of $\bfs$ and $\bft$, and define/recall
	\begin{align*}
	M&=\sum_{i\in X} s_i=\sum_{j\in Y} t_j; \\
	S_k&=\sum_{i\in X} (s_i)_k; \ \ T_k=\sum_{j\in Y} (t_j)_k \quad \mbox{for all $k\ge 2$,}
	\end{align*}
	where $(x)_k=\prod_{i=0}^{k-1}(x-i)$ denotes the $k$-th falling factorial.
	
	Let $\M(\bfs,\bft)$ denote the set of bipartite multigraphs with bipartition $(X,Y)$ and bipartite degree sequence $(\bfs,\bft)$.   
	We will show that following algorithm, \MatrixGen, is a uniform sampler for $\M(\bfs,\bft)$.
	
		\medskip
	\begin{algorithm}[H]
		{\bf{procedure}}
		{\MatrixGen}(\bfs,\bft)\\
		{\bf with probability} $1-\rho$ {\bf do}\\
		\quad G:=\SimpleBiGen(\bfs,\bft);\\
		\quad\Gen(G);\\
		{\bf else}  \\
		\quad \Brute(\bfs,\bft)\\
		{\bf end}
	\end{algorithm}
		\medskip
	 The parameter $\rho$ is a function of the input degree sequences $\bfs$ and $\bft$, and will be specified later. The choice of $\rho$ is based mainly on complexity issues, and for the usable range of parameters, it is very small. Thus,  \MatrixGen\  usually calls \SimpleBiGen~\cite{arman19}, which is a Las Vegas algorithm generating a random simple bipartite graph  with the given degree sequence, uniformly at random.
	\begin{thm}[\cite{arman19}] \lab{thm:simple}
		Assume $(\bfs,\bft)$ is a bipartite degree sequence, and $\Delta$ denotes the maximal degree. Algorithm \SimpleBiGen\ generates a uniformly random simple bipartite graph with degree sequence $(\bfs,\bft)$. Moreover, provided $\Delta^4=O(M)$, the expected runtime of \SimpleBiGen\ is $O(M)$  and its space complexity is always $O(mn)$. 
	\end{thm}
	After \SimpleBiGen\ produces a bipartite graph $G$, procedure \Gen\  employs a set of switching operations, described below,  to add  multiple edges to $G$, outputting a random multigraph for which the  total multiplicity of multiple edges does not exceed some pre-specified value $\totmultcut-1$, which is a function of the input degree sequence. On the other hand, \Brute\ uses recursion and a carefully designed data structure to generate the remaining multigraphs, i.e.\ those with   total multiplicity at least $\totmultcut$. Both \Gen\ and \Brute\ have the possibility of essentially failing, at which point they cause algorithm \MatrixGen\ to restart from scratch. This possibility, together with the use of \SimpleBiGen, is responsible for the Las Vegas nature of the algorithm \MatrixGen. Each restart has a positive probability of producing an output, so the algorithm terminates eventually.
	
\begin{remark}
		 If $\Delta=1$, then \MatrixGen\ should generate a random matching and can be implemented by just calling \SimpleBiGen\ once. For this case time and space complexity of \MatrixGen\ is same as of \SimpleBiGen\, and Theorem~\ref{thm:matrix} follows from Theorem~\ref{thm:simple}. A similar comment applies to Theorem~\ref{thm:multigraph}. For the rest of the paper we assume that $\Delta\geq 2$. 
\end{remark}
	
We present the description of the less important  procedure  \Brute\ in Section~\ref{sec:Brute}. Here we focus on the main component, procedure \Gen. A multiple edge of multiplicity $j$ is a set of $j$ edges sharing the same two end vertices. For every integer $t\ge 2$, the $t$-switching, defined  as follows, is used to add a multiple edge with multiplicity $t$. 
	\begin{definition}[$t$-switching]\lab{d:switching}
		Let $u_1, v_1, \ldots, u_{t+1},v_{t+1}$ be a set of $2(t+1)$  distinct vertices such that 
		\begin{itemize}
			\item $u_1\in X$, $v_1\in Y$, and for all $2\leq i \leq t+1$, $u_i\in Y$ and $v_i\in X$;
			\item $u_1 u_i$ and $v_1 v_i$ are single edges  for all $2\le i\le t+1$;
			\item for every $2\le i\le t+1$, $u_i$ is not adjacent to $v_i$.
		\end{itemize}
		The $t$-switching replaces edges $u_1 u_i$ and $v_1 v_i$, for each $2\le i\le t+1$, by $u_iv_i$ for each $2\le i\le t+1$ and a multiple edge with multiplicity $t$ between $u_1$ and $v_1$. 
	\end{definition}
	See Figure~\ref{f:2-way} for an illustration of a 2-switching. We note that the transition in the Diaconis-Gangolli Markov chain  modifies  the entries of a $2\times 2$ submatrix, whereas the $t$-switching, expressed in the matrix version, modifies the entries of a $(t+1)\times (t+1)$ submatrix.
	\begin{figure}[H]
		\begin{center}
			\includegraphics[trim={2.5cm 21cm 0cm 2.5cm},clip]{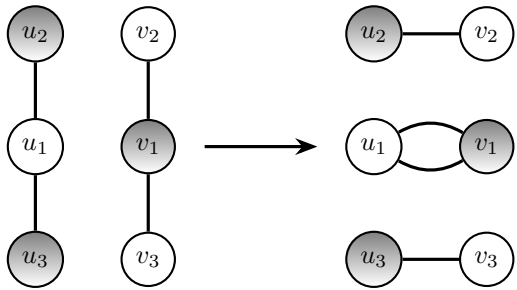}
			\caption{A 2-switching, vertices of $X$ are white, and those of $Y$ are shaded.}
			\label{f:2-way}
		\end{center}	
	\end{figure}
	
	Let $\bfm=(m_1,m_2,\ldots,m_{\Delta})$ be a vector of nonnegative integers. Let $\H_{\bfm}$ denote the set of bipartite multigraphs in $\M(\bfs,\bft)$ where the number of multiple edges of multiplicity $i$ equals $m_i$ for every $2\le i\le \Delta$. (We omit the notation $(\bfs,\bft)$ in $\H_{\bfm}$ since these vectors are fixed during the algorithm, whereas  $\bfm$ undergoes changes.) We set $m_1=0$ always as it is not needed. A multigraph in $\H(\bfm)$ has $\sum_{i=2}^{\Delta} m_i$ multiple edges in total, and has $S(\bfm):=\sum_{i=2}^{\Delta} i m_i$ edges that are contained in multiple edges. We say that $\bfm$ is the {\em stratum index} of the multigraphs in $\H_\bfm$.
	
	A set of parameters
	\bel{parameters}
	\totmultcut,\quad \rho,\quad \beta_{\bfm}, \quad \UB_k(\bfm),\quad \LB_k(\bfm),\quad \underline{b}_k(\bfm,i)
	\ee
	will be specified later. Here $\totmultcut$ is chosen to be much greater than the total number edges contained in multiple edges in a ``typical'' multigraph in $\M(\bfs,\bft)$, whereas $\rho>0$ is an upper bound on the probability that a uniformly random multigraph has more than $\totmultcut$ edges contained in multiple edges.
	
	The procedure   \Gen\ adds double edges, one at each step, and then adds triple edges, and so on, using the corresponding $t$-switchings.  Formalising this 
	description, we define a binary relation $\prec$ on $\bfm\in \bbN^{\Delta}$ as follows. 
	Given $\bfm\in \bbN^{\Delta}$,  let $\ell(\bfm)$ be the  greatest  index $k$ such that $m_k>0$. If $\bfm$ is the zero vector, set $\ell(\bfm)=2$.  We  write that $\bfm\prec \bfm'$ if $\bfm_i=\bfm'_i$ for all $i<\ell(\bfm)$ and $\bfm$ is smaller than $\bfm^\prime$ in lexicographic order.  Note that $\prec $ is not reflexive in our definition. 
	In each step,  \Gen\  either performs a switching and creates a multiple edge, or terminates by either outputting the current graph or performing a rejection.

	The action of \Gen\ can be described as a  Markov chain whose states  are multigraphs $G_0,G_1,\ldots$, where $G_i$ has $i$ multiple edges, together with two terminal states, signifying output and rejection. The random variable $\bfm(G_i)$  defines a random process on the vectors $\bfm \in \bbN^{\Delta}$. This process undergoes transitions governed by another Markov chain, called  the stratum index Markov chain, whose non-terminal states are stratum indices. The stratum index chain determines $\bfm(G_0),\bfm(G_1),\ldots$ until $\Gen$ terminates, which can occur  either by rejection during a switching operation, or when the stratum index chain itself terminates.   It is convenient to decompose each step of  \Gen\ into sub-steps at two `levels'. The first  sub-step, at the upper level, takes one step of the stratum index Markov chain, which moves either to a termination state (rejection, or output  the current multigraph $G_i$) or  to the next stratum index. The second sub-step, at the lower level, is invoked if non-termination occurred in the first step, and   either performs a switching  that converts the current multigraph $G_i$ to a multigraph $G_{i+1}$ with the new stratum index, or performs a rejection.
	
	We now explain the parameter $\beta_{\bfm}$. This  will be chosen as a rather tight upper bound on 
	\bel{Hmplus}
	\frac{|\H_{\bfm}^+|}{|\H_{\bfm}|}, \quad\mbox{where}\ \H_{\bfm}^+=\bigcup_{\scriptsize\begin{array}{c}\bfm': \bfm\prec \bfm' \\ S(\bfm')<\totmultcut\end{array}} \H_{\bfm'},
	\ee
	recalling that $S(\bfm')=\sum_{i=2}^{\Delta}im'_i$, which is the total number of edges contained in multiple edges in a multigraph from $\H_{\bfm'}$.
	The transition probabilities in the stratum index Markov chain are  determined using  $\beta_{\bfm}$. If $\beta_{\bfm}$ were defined exactly equal to the ratio $|\H_{\bfm}^+|/|\H_{\bfm}|$, then no rejection  state would be needed in this chain, because at each step a transition would be chosen with exactly the correct probability. Since this is not the case, a rejection scheme (which we call $\beta$-rej) is necessary.  
	There is a superpolynomial number of possible $\bfm$, and thus  computing all $\beta_{\bfm}$ would take superpolynomial time. Instead, the algorithm only pre-computes $\beta_{\bfm}$ where $\ell(\bfm)=2$ and $m_2<\totmultcut/2$. For all other $\bfm$, an explicit formula for $\beta_{\bfm}$ is used instead. 
 	
	Several types of rejections are used in the switching step.  These rejections are guided by parameters $\UB_k(\bfm)$ and  $\LB_k(\bfm), \underline{b}_k(\bfm,i)$ respectively. 
		
	In~\cite{mckay90,gao17,gao18}, the analogue of the stratum index Markov chain was a much simpler Markov chain that only had one possible stratum for output. The   difficulties encountered in the present work stem mainly from having to permit outputs from arbitrary strata, where the distribution of the final  stratum  is too difficult to compute. This necessitates innovative design of the stratum index Markov chain. The switching sub-step of \Gen\  is basically a straightforward extension of the switching step  in~\cite{mckay90,gao17,gao18} to handle arbitrarily high multiplicities. However, applying the rejection scheme from~\cite{mckay90,gao17,gao18}, or more precisely extending it to high multiplicity edges, would  yield superpolynomial computation time  due to rejection steps involving edges whose multiplicity can tend to infinity.  Instead we   use a fast technique called ``incremental relaxation'' recently developed in~\cite{arman19} to implement the switching step. This permits \Gen\ to run in linear time under the assumptions of Theorem~\ref{thm:matrix}.  
	
	\begin{remark}
		The assumption that $(\bfs, \bft)$ is graphical in Theorem~\ref{thm:matrix} is only needed to ensure that we can call \SimpleBiGen. 
		This assumption could be omitted if we modify the algorithm \MatrixGen\ as follows. If $(\bfs,\bft)$ is not graphical, one could use the configuration model to generate a bipartite multigraph with degrees $(\bfs, \bft)$. An easy but quite cumbersome alteration of our algorithm can be made, so that the algorithm probabilistically decides whether to add  multiple edges, as in \Gen\, or to remove them using a similar process.  By assuming graphicality, we avoid such complications in the description of the algorithm. This is without sacrificing the power of the main result, since  in order to obtain   expected linear run time for  \MatrixGen, in Theorem~\ref{thm:matrix} we made the assumption that $5\Delta^4<M$, which already implies graphicality.
	\end{remark}
		
	We give the full description of \Gen\ in Section~\ref{sec:algorithm}. The parameters appearing in~\eqn{parameters} will be specified and explained in Section~\ref{sec:parameters}. We prove that \MatrixGen\ is a uniform sampler  and estimate time and space complexity in Section~\ref{sec:GenUniTime}. In Section~\ref{sec:Brute} we provide a description of \Brute, prove that it is a uniform sampler and estimate its complexity. Finally, in Section~\ref{sec:proofmain}, we prove Theorem~\ref{thm:matrix}.

	
	\section{Procedure \Gen}
	\lab{sec:algorithm}
	Here we define \Gen. The input of \Gen\ is a graph $G$ chosen uniformly at random from all simple bipartite graphs with degree sequence $(\bfs,\bft)$. Each iteration of the  loop in  \Gen\ takes a step in the stratum index Markov chain, and then performs the required switching (or rejects and restarts).
	
	For a vector $\bfm$, recall that $S(\bfm)=\sum_{i=2}^\Delta im_i$. Throughout this paper we use  $\bfe_t$  to denote the   unit  vector 
	which has 0   everywhere except for a 1  in the $t$-th coordinate. The reader should notice that the variable $t$ in {\Gen} keeps track of the minimum possible multiplicity of the remaining edges to be added.
	
	\medskip
	\begin{algorithm}[H]
		{\bf{procedure}}
		{\Gen}(G)\\
		$\textbf{m}=(0,\ldots , 0)$\;
		$t=2$\;
		\While{$S(\textbf{m}) < \totmultcut$}{
			\textbf{output} $G$ with probability $\frac{1}{1+\beta_{\textbf{m}}}$\;
			choose $s\in\{t, \ldots, \Delta\}$ with probability $\displaystyle\frac{\UB_s{(\bfm)}}{\LB_s{(\bfm+\bfe_s)}}\cdot \frac{(1+\beta_{\textbf{m}+\bfe_{s}})}{  \beta_{\textbf{m}}}$; else {$\beta$-\textbf{reject}} \;
			Set $\textbf{m}^{\prime}=\textbf{m}+\bfe_{s}$\;
			Select a random $s$-switching $S$ that can be performed on $G\in \mathcal{H}_{\textbf{m}}$ to obtain $G^{\prime} \in \mathcal{H}_{\bfm'}$\;
			{f-\textbf{reject}} with probability $1-f_{s}(G)\slash\overline{f}_s(\bfm)$ \;
			{b-\textbf{reject}} with probability $1-\underline{b}_s(\bfm)\slash b_s(G^\prime,S)$ \;
			$G \gets G^{\prime}$\;
			$\textbf{m}\gets \textbf{m}^\prime$\;
			$t\gets s$;
		}
	\end{algorithm}
	\medskip
	Each of the three rejections causes a restart of \MatrixGen\ .
	For the following lemma, we assume the parameters in~\eqn{parameters} are specified as in Section~\ref{sec:parameters}. The lemma ensures that \Gen\ is well defined in the sense that all probabilities called for are between 0 and 1. Its  proof can be found in Appendix A1.
	\begin{lemma}\label{lem:probability} For all $\textbf{m}$ with $S(\bfm)< \totmultcut$ and $\ell(\textbf{m})\geq 2$
		$$\sum _{s=2}^{\Delta}\frac{\UB_s{(\bfm)}}{\LB_s{(\bfm+\bfe_s)}}\cdot \frac{1+\beta_{\bfm+\bfe_{s}}}{\beta_{\textbf{m}}}\leq 1.$$ 
	\end{lemma}
	
	Roughly, in each iteration of \Gen\, a new stratum indexed by $\bfm'=\bfm+\bfe_s$ is chosen and then \Gen\ randomly selects an $s$-switching $S$ which converts $G$ to some graph  $G'\in\mathcal{H}_{\bfm+\bfe_s}$. After this, based  on what are essentially subgraph counts in $G'$, \Gen\ computes two probabilities, $p_1$ and $p_2$. It performs an f-rejection with probability $p_1$ and then a b-rejection with probability $p_2$. If any rejection occurs then the whole algorithm \MatrixGen\ restarts. Otherwise, \Gen\ returns $G'$. 
	
	We first explain  f-rejection. Given $s$ and $G$, let $f_s(G)$ be the number of $s$-switchings that can be performed on $G$ and let $\UB_s(\bfm)$ be a uniform upper bound on $f_s(G)$ for $G\in \mathcal{H}_{\bfm}$.  Let $p_1=1-f_s(G)\slash \UB_s(\bfm)$, i.e.\   f-rejection is performed  with probability $1-f_s(G)\slash \UB_s(\bfm)$.  By choosing an appropriate $\UB_s(\bfm)$, we can implement f-rejection in a way that avoids computation of $f_s(G)$. See Section~\ref{sec:GenUniTime} for details.    
	
	Next we explain the more complicated b-rejection. 
	The scheme that does the job of b-rejection in~\cite{mckay90},~\cite{gao17} and~\cite{gao18} requires computing the number of $s$-switchings that can produce $G'$, which can take long time if $s$ is large. We use instead the new scheme in~\cite{arman19}, which takes advantage of the fact that  $G'$ is generated from the random switching uniformly at random subject to a certain set of constraints that are derived from the switching. Computing the number of $s$-switchings that produce $G'$ is equivalent to counting the number of ways to place the appropriate set of constraints on $G'$ simultaneously. Instead of doing so, the new b-rejection scheme takes $G'$ and the set of constraints at the input, and computes the number of ways to relax just one of the constraints, uses this to obtain a uniformly random multigraph subject to one less constraint, and  then iterates. This technique is called {\em incremental relaxation}, and is described in~\cite{arman19} in a more general setting. 	 	
	For an  $s$-switching $S$ from    $G \in \mathcal{H}_{\bfm}$ to $G' \in \mathcal{H}_{\bfm+\bfe_s}$, we write $G=\inS$ and $G'=\outS$.
	
	Such a switching is equivalent to an \emph{anchored} graph $(G',V(S))$, where $V(S)=(u_1,v_1, \ldots, u_{s+1},v_{s+1})$ is an ordered subset of vertices of $G'$, with $u_1\in X$ and $v_1\in Y$, $u_i\in Y$ and $v_i\in X$ for all $i \geq 2$, that also satisfies the following set of adjacency constraints:
	
	\begin{itemize}
		\item $u_1v_1$ is an edge of multiplicity $s$,
		
		\item $u_iv_i$ are single edges for $i\in[2,s+1]$ 
		
		\item there are no edges between $u_1$ and $v_i$ or $v_1$ and $u_i$ for $i\in[2,s+1]$.
	\end{itemize}
	The equivalence is obtained by noting that an application of an $s$-switching as in Definition~\ref{d:switching}, paying attention to the names of the vertices, determines an anchored graph as above, and vice versa.  
	
	For each $s$-switching $S$ and $1\leq i\leq  s+1$, let $V_{i}(S)=(u_1,v_1, \ldots, u_{i},v_{i})$ and let $V_0(S)=\emptyset$.
	Given a multigraph $H'$, and given $1\leq i \leq s+1$ and  $V'_i=(u'_1,v'_1 , \ldots, u'_{i}, v'_{i})\subseteq V(H')$, we say that $V'_i$ is a valid $i$-subset in $H'$ (with respect to $s$-switchings) if there exists an $s$-switching $S'$ which converts some multigraph to $H'$, and for which $V_i(S')=V'_i$ (e.g.\ consider $V'_1=(u_1, v_1)$ if $H'=G'$).
	For $1\leq i \leq s$, a valid $i$-subset  $V'_i=(u'_1,v'_1 , \ldots, u'_{i}, v'_{i})$ in $H'$ and a simple ordered edge $(u,v)$ in $H'$, we say that $(u,v)$ is \emph{switching compatible} with $V'_i$ if there exists an $s$-switching $S'$ which converts some multigraph to $H'$, such that $V_{i+1}(S')$ equals $V'_i+(u,v):=(u'_1,v'_1 , \ldots, u'_{i}, v'_{i},u,v)$. 
	
	Define $b_{s}(G^\prime, V_0)$ to be the number of valid $1$-subsets in $G'$,  which we will prove is equal to the number of multiple edges of multiplicity $s$ in $G'$.  (See Lemma~\ref{lem:sub-b-bound} below).  For $1 \leq i \leq s$ and a valid $i$-subset $V_i$ in $G'$, define $b_s(G',V_i)$ to be the number of simple ordered edges that are switching compatible with $V_i$ in $G'$. 
	
	The parameters $\underline{b}_{s}(\bfm, i)$ for $0\leq i \leq s$, which are to be specified in Section~\ref{sec:parameters}, are chosen to be  uniform lower bounds for the respective $b_s(G',V_i(S))$, over all $G'\in \H_{\bfm}$ and all valid $s$-switchings $S$ that produce $G'$.   
	We also set
	\be
	\underline{b}_{s}(\bfm)=\prod_{i=0}^{{s}}\underline{b}_s(\bfm,i), \lab{def_b}
	\ee
	which obviously is a uniform lower bound for 
	\[
	b_s(G',S)=\prod_{i=0}^{s} b_s(G',V_i(S)).
	\]
	The b-rejection scheme computes each $b_s(G',V_i(S))$ and sets $p_2=1-\underline{b}_{s}(\bfm)/b_s(G',S)$, and performs a b-rejection with probability $1-\underline{b}_{s}(\bfm)/b_s(G',S)$. Computation of $b_s(G',V_i(S))$ can be done rapidly;  see details in Section~\ref{sec:GenUniTime}. 	
	
	Without introducing all the general terminology of~\cite{arman19}, we remark for those familiar with~\cite{arman19} that the constraints $C_i$ in~\cite[Section 3]{arman19} correspond to the constraints associated with the edges joining vertices of $V_i(S)$. Seen in this way,~\cite[Corollary 6]{arman19} implies that 
	if $G$ is uniformly random in $\mathcal{H}_{\textbf{m}}$ then the switching creates  $G^{\prime}$   uniformly random in $\mathcal{H}_{\textbf{m}+\bfe_s}$. However, we do not rely on this fact directly in the present paper, and instead  derive it in a more precise form in Section~\ref{sec:GenUniTime}.

	\def\ms{\medskip}
	\section{Parameter set-up}
	\lab{sec:parameters}
	Here we define the parameters involved in the algorithm.  Recall that $(\bfs,\bft)$ is a bipartite degree sequence that satisfies $\Delta\geq 2$ and  $5\Delta^4< M$.
	
	 The proofs of Lemmas~\ref{lem:sub-b-bound}--~\ref{lemma:tail} that are stated in this section are based on straightforward inclusion-exclusion arguments and calculus. The proofs are presented in Appendix A1. \ms 
	
	\noindent{\em Choice of $\totmultcut$}. 
	\ss
	
	If  possible,  choose integer $\totmultcut>7$ such that for $\eps=1-(\totmultcut+4\Delta^2)/M$,  we have $\epsilon=\Omega(1)$ and
	$$
	\eps^{3}> 4\Delta^4/M \qquad \mbox{and}\qquad 2\totmultcut(\totmultcut -\Delta^2-\Delta^3)\geq   (\totmultcut+4\Delta^2)^2.
	$$ 
	Otherwise set $\totmultcut=0$.

	Note that $\totmultcut\ne0$ provided $M$ is large enough, and also that if $\Delta$ is bounded we can choose $\totmultcut\sim 3M/4$  so that $\eps\sim 1/4$. 
	\ms
	
	\noindent{\em Choice of $\underline{b}_k(\bfm,i)$, $\underline{b}_k(\bfm)$ and $\overline{f}_k(\bf m)$.}\ss
	
	For $k\geq 3$ and  $1\leq i\leq k$ define 
	$$ \underline{b}_{k}(\bfm,i)=\epsilon M, \qquad 	\underline{b}_{k}(\bfm,0)=m_k. $$
	For $1\leq i\leq 2$ define 
	\begin{align*}
	\underline{b}_{2}(\bfm,i)&=M\left(1-\frac{S(\bfm)+2i\Delta+2\Delta^2}{M}\right),\quad 
	\underline{b}_{2}(\bfm,0)=m_2.
	\end{align*}
	
	Finally, 
	\begin{align*}
	\underline{b}_k(\bfm)&=\prod_{i=0}^{k}\underline{b}_k(\bfm,i)=m_k\prod_{i=1}^{k}\underline{b}_k(\bfm,i),\quad
	\overline{f}_{k}(\bfm)=S_kT_k.
	\end{align*}
	Recall that for $G\in \mathcal{H}_{\textbf{m}}$, the value of $f_k(G)$ is the number of possible $k$-switchings $S$ that can be performed on $G$. Next we verify that the parameters specified above are uniform lower and upper bounds for $b_k(G,S)$ and $f_k(G)$ respectively. 
	
	\begin{lemma}\lab{lem:sub-b-bound}
		Assume that $G\in \mathcal{H}_{\bfm}$ is a multigraph with $S(\bfm)\leq \totmultcut$. Let  $S$ 
		be a $k$-switching that produces $G$. Then, for all $1 \leq i\leq k$,
		$$\underline{b}_k(\bfm,i)\leq b_k(G,V_i(S)) \leq M,$$ 
		and for  $i=0$ we have $b_k(G,V_0(S))=\LB_k(\bfm,0)$.\end{lemma}

	\begin{lemma}\label{lemma:tbounds}
		Let $G\in \mathcal{H}_{\textbf{m}}$ be such that $S(\textbf{m})\leq \totmultcut$ and  let 
		$S$ be a $k$-switching that produces $G$. Then
		$$f_k (G) \leq \overline{f}_k(\bfm) \;\;\; \text{and} \;\;\;  \underline{b}_{k}(\bfm) \leq b_k(G,S) \leq m_kM^k.$$
		Moreover, for  $k=2$ we also have 
		$$f_{2}(G)\geq S_2T_2\left( 1- \frac{\Delta(S_2+T_2)(2S(\bfm)+2\Delta+1.5\Delta^2)}{S_2T_2}\right).$$
	\end{lemma}
	
	In the case when $\totmultcut=(1-\epsilon)M-4\Delta^2$, for every $\bfm$ with $S(\bfm)<\totmultcut$ and $2\le k\le \Delta$ we have $\LB_{k}(\bfm+\bfe_k)\geq (m_k+1)M^k\epsilon^k$ and hence
	\begin{equation}\label{ineq:fb}\frac{\UB_k(\bfm)}{\LB_k(\bfm+\bfe_k)}\leq \frac{S_kT_k}{(m_k+1)M^{k}\epsilon^k}.
	\end{equation}
	The last inequality is often used in the rest of the paper.

	\ms
	
	\noindent{\em Choice of $\beta_\bfm$.}
	\ss
	
	Define
	\begin{align*}
	S_{\totmultcut}=    \{\bfm: S(\bfm)= \totmultcut\},\qquad S_{\totmultcut}^+ =  \{\bfm: S(\bfm)\ge \totmultcut\},\qquad S_{\totmultcut}^- =    \{\bfm: S(\bfm)< \totmultcut\} .
	\end{align*}    
	
	If $\bfm\in S_{\totmultcut}^+$, set $\beta_{\textbf{m}}=-1$. If  $\bfm\in S_{\totmultcut}^-$ and $\ell(\textbf{m})=\ell\geq 3$,   set $$\beta_{\textbf{m}}=\frac{4\Delta^{2\ell-2}}{M^{\ell-2} \eps^{\ell}}.$$
	For each remaining sequence $\textbf{m}=(0, m_2, 0, \ldots,0)\in S^-_{\totmultcut}$, inductively for decreasing values of $m_2$, set 
	$$\beta_{\textbf{m}}=\sum_{i=2}^\Delta\frac{\overline{f}_i(\bfm)}{\underline{b}_i(\bfm+\bfe_i)}(1+\beta_{\bfm+\bfe_i}).$$
	Note that this definition of $\beta_{\bfm}$ ensures that \Gen\ will finish either with a restart or with outputing some multigraph $G$. Indeed, the  only situation that could potentially cause a problem is that an iteration of the while loop of \Gen\ starts with some $G\in\mathcal{H}_{\bfm}$ with $S(\bfm)<\totmultcut$, and generates some $G'\in \mathcal{H}_{\bfm+\bfe_k}$ with $S(\bfm+\bfe_k)\geq \totmultcut$. However, this cannot happen, since in this case $\beta_{\bfm+\bfe_k}=-1$ and thus, for a  given $G\in \mathcal{H}_{\bfm}$, \Gen\ chooses $s=k$ with probability equal to 0.

	Recall the definition of $\mathcal{H}^{+}_{\textbf{m}}$ from~\eqn{Hmplus}.
	\begin{lemma}\lab{lem:beta} For all $\textbf{m}$ with $S(\textbf{m})< \totmultcut$ we have $$\beta_{\textbf{m}}\geq \frac{|\mathcal{H}^{+}_{\textbf{m}}|}{|\mathcal{H}_{\textbf{m}}|}.$$
	\end{lemma}
	\ms
	
	\noindent{\em Choice of $\rho$.}
	\ss
	
	If $\totmultcut>0$, set
	$$
	\rho=\frac{B \frac{1}{1+\beta_{\textbf{0}}}}{1+B\frac{1}{1+\beta_{\textbf{0}}}}, \quad\mbox{where~} B=4\left(\frac{3}{2(1-\epsilon)^2}+\frac{3}{4(1-\epsilon)^4}\right)^{\epsilon M\slash 2} \left(\frac{\Delta^2e}{\epsilon^2(1-\epsilon)(\totmultcut-7)}\right)^{\totmultcut-7}.
	$$
	If $\totmultcut=0$, set $\rho=1$.
	\begin{remark}\label{remark:B}
		From the definition of $B$ and $\totmultcut$ it follows that $B=M^{-\Omega(M)}$, provided $M$ is large enough.
	\end{remark}
	\begin{lemma}\label{lemma:tail}
		If $\totmultcut>0$, we have 
		$$\frac{|\cup_{\textbf{m}\in S^+_{\totmultcut}}\mathcal{H}_{\textbf{m}}|}{|\mathcal{H}_{\textbf{0}}|}\leq B. $$
		\end{lemma}
	
	We note immediately that $\totmultcut$ equals 0 only for finitely many $M$. In this case $\rho=1$, and hence only \Brute\ is called in \MatrixGen. In Section~\ref{sec:Brute} it is shown that \Brute\ is uniform generator and has a constant expected running time. Hence for the rest of the paper we consider only the case $\totmultcut>0$.
	
	\section{ Uniformity, time and space complexity of \Gen}
	\lab{sec:GenUniTime}
 In this section we prove that \Gen\  is a uniform sampler and estimate its expected run time and space complexity. 
	
	\begin{thm}\lab{thm:uniformgen}
		 Assume that $G$ is distributed uniformly in $\mathcal{H}_0$. Then \Gen$(G)$ generates every bipartite multigraph with bipartite degree sequence $(\textbf{s}, \textbf{t})$ and total multiplicity $S(\bfm)<t_0$ with probability equal to $\frac{1}{(1+\beta_o)|\mathcal{H}_0|}$. 
	\end{thm}
	\begin{proof}    
		We say that a multigraph $H$ was {\em reached} in \Gen\ if a switching creating $H$ was selected in a switching step, and was not rejected.   Let $G_0=G$, and $G_t$ denote the multigraph reached after $t$ switching steps. If  \Gen\ terminates before completing $t$ non-rejected switchings,  let $G_t=\emptyset$. We will prove by induction on $t$ that, conditional on $G_t\in\H_{\bfm}$, $G_t$ is uniformly distributed in $\H_{\bfm}$. Assume $t\ge 0$ and the inductive statement holds for $t$. Let $G_t\in \H_{\bfm}$ for some $\bfm\in S^-(\totmultcut)$.  Then, there exists $\sigma_{\bfm}$ such that the probability that $G$ is reached after $t$ switching steps is equal to $\sigma_{\bfm}$, for every $G\in\H_{\bfm}$. To establish the inductive step we prove the following. 
		\begin{claim}
			For every $k\geq \ell(\bfm)$ and every  $G\in \H_{\bfm+\bfe_k}$ such that $\bfm+\bfe_k\in S^-(\totmultcut)$,
			$$\mathbb{P}(G_{t+1}=G)=\sigma_{\bfm}\frac{1+\beta_{\textbf{m}+\bfe_{k}}}{1+\beta_{\textbf{m}}}.$$
			
		\end{claim}
		\noindent{\bf Proof of Claim.}
		By the description of \Gen,
		\[
		\pr(G_{t+1}=G)=\sum_{S\,:\,\outS=G} \pr(\inS   \mbox{ is reached and $S$ is selected and not rejected}) 
		\]
		where the summation is restricted to $k$-switchings $S$. For such an  $S$ with $G=\outS$, the probability of $\inS$ being reached is   $\sigma_{\bfm}$, and conditional upon that, the probability of not outputting  at the beginning of step $t+1$  is $\beta_{\bfm}/(1+\beta_{\bfm})$. 
		Conditional on that, the probability that the stratum index Markov chain transitions from $\H_{\bfm}$ to $\H_{\bfm+\bfe_k}$ is $(\overline{f}_k(\bfm)/\underline{b}_k(\bfm+\bfe_k))(1+\beta_{\bfm+\bfe_k})/\beta_{\bfm}$. Conditional on that, the    probability of selecting the particular switching $S$ is $1/f_k(\inS)$.
		Conditional on that, the probability that $S$ is not f- or b-rejected  is 
		\[
		\frac{f_k(\inS)}{\overline{f}_k(\bfm)}\frac{\underline{b}_k(\bfm+\bfe_k)}{b_k(G,S)}=\frac{f_k(\inS)\underline{b}_k(\bfm+\bfe_k)}{\overline{f}_k(\bfm)}\prod_{i=0}^{k}\frac{1}{b_k(G, V_i(S))}.
		\]  
		Taking the product of all terms and summing over appropriate $S$, we have
		\[
		\pr(G_{t+1}=G)=\sigma_{\bfm}\frac{1+\beta_{\bfm+\bfe_k}}{1+\beta_{\bfm}} \sum_{S\,:\,G=\outS} \prod_{i=0}^{k}\frac{1}{b_k(G,V_i(S))} 
		\]
		with the summation again restricted to $k$-switchings $S$.
		It only remains to prove that the sum of products above is equal to 1. This is easily verified by showing, by reverse induction on $j$, that for any valid $  V'_j$ 
		$$
		\sum_{S\,:\,G=\outS,\, V_j(S)=V'_j}\  \prod_{i=j}^{k}\frac{1}{b_k(G, V_i(S))}  = 1.
		$$
		The case $j=k+1$ is trivial as the product is empty, and the inductive step follows from the fact that, by definition, there are exactly $b_k(G,V_{j-1})$ switching-compatible choices for $V_j$, given $V_{j-1}$.  \qed
		\ss
		
		Note that, in the terminology of~\cite{arman19}, the inductive step in the above proof corresponds to the incremental relaxation from what is essentially a uniformly chosen  random multigraph   anchored at $V_j$ to a similar  one anchored at $V_{j-1}$.	
		
		The claim implies that
		$$
		\sigma_{\bfm+\bfe_k}\frac{1}{1+\beta_{\textbf{m}+\bfe_{k}}}=\sigma_{\bfm}\frac{1}{1+\beta_{\textbf{m}}}.
		$$ 
		Next we prove that for every $\bfm$ and $k$ such that $\bfm+\bfe_k\in S^-(\totmultcut)$, and every $G \in \mathcal{H}_{\textbf{m}+\textbf{e}_k}$ and $G' \in \mathcal{H}_{\textbf{m}}$, the probabilities that $G$ and $G'$ are outputted are equal. From this it follows that \Gen\ outputs every multigraph in  $\mathcal{H}_{\textbf{m}}$ with $\bfm\in S^-(\totmultcut)$ with the same probability. We have 
		\begin{align*}
		&\mathbb{P}(\text{output} \; G )=\mathbb{P}(G_t=G)\frac{1}{1+\beta_{\textbf{m}+e_{k}}}=\sigma_{\bfm+\bfe_k}\frac{1}{1+\beta_{\textbf{m}+\bfe_{k}}}\\
		=&\sigma_{\bfm}\frac{1}{1+\beta_{\textbf{m}}}=\mathbb{P}(G_{t-1}=G^\prime )\frac{1}{1+\beta_{\textbf{m}}}=	\mathbb{P}(\text{output} \; G').
		\end{align*} 
		
		 Finally, for any simple bipartite $G$ the probability that \Gen\ outputs $G$ is equal to  $\frac{1}{(1+\beta_0)|\mathcal{H}_{0}|}$.

\qed \end{proof} 

	\begin{thm}\label{thm:GenTimeSpace}
	\Gen\ has time complexity $O(M)$ and space complexity $O(mn\log (\Delta+1))$.	
	\end{thm}
	\begin{proof}	
	We use appropriate data structures to store the set of edges (adjacency matrix without initialisation), and the positions of the multiple edges so that look-up operations take constant time.
		
	Pre-calculation of $\beta_{\textbf{m}}$ can be done in time $O(M)$. Only those with   $\ell(\bfm)\le 2$  need to be treated, since the others are specified in Section~\ref{sec:parameters}. All of the moments $S_{i}$ and $T_i$ can be computed in advance by first computing the frequencies of the degrees, and then calculating each moment as a running sum. Overall, the time required for this is  $O(n+\Delta^2)=O(M)$ since all degrees are positive.  After that, those $\beta_{\bfm}$ with $\ell(\bfm)=2$ can be calculated inductively in time $O(\totmultcut)=O(M)$, since in the summation defining them, all terms with $i>2$ are independent of $\bfm$ and can thus be pre-computed.
	
	Next consider the f-rejection step. We do not need to evaluate $f_{s}(G)$: we can choose a pair of $s$-stars, independently, and uniformly at random, and reject  if this pair of stars does not designate a valid $s$-switching. Since $\overline{f}_s(\bfm)$ is exactly the number of pairs of $s$-stars allowing repetition, the probability of rejection is exactly $1-f_s(G)/\overline{f}_s(\bfm)$, as desired. 
	
	Finally, to calculate the value of $b_{s}(G')$,  we  know that $b_s(G^\prime, V_0)=m_s+1$ and   need additionally to find  $b_{s}(G^\prime,V_i(S))$ for  all $i\in[s-1]$.  Each $b_{s}(G^\prime,V_i(S))$ is equal to the number of simple edges in $G'$ minus $|X_i(S)|$, where $X_i(S)$ is the set of simple edges that have at least one endpoint in $V_i(S)$, or in the neigbourhood of $\{u_1,v_1\}$. The initial set $X_1(S)$ can be determined by examining the $2$-neighbourhood of the multiple edge $u_1v_1$, which can be done in time $O(\Delta^2)$. Each subsequent $X_i(S)$ is obtained from $X_{i-1}(S)$ by adding edges that are incident with at least one of vertices $u_i$ or $v_i$. Hence it takes $O(\Delta)$ time to  obtain $X_i(S)$ from $X_{i-1}(S)$. This update has to be done at most $\Delta$ times, making $O(\Delta^2)$ time in total to compute $b_{s}(G')$. Assembling these observations, we conclude that each iteration of the {\bf while} loop (i.e.\ switching step) in \Gen\ requires time $O(\Delta^2)$.
	
	Assuming that $\Gen$ creates at most $t_0:=\max\{24, 4 S_2T_2/\epsilon^2M^2\}$ double edges, it requires $O(t_0\Delta^2)=O(\Delta^4)$ computation time. The following lemma shows that with sufficiently high probability \Gen\ does not create so many double edges.
	\begin{lemma}\label{lemma:largem_0}
		The probability that \Gen\ reaches a graph in $\mathcal{H}_{\bfm}$ with $m_2> 3 S_2T_2/\epsilon^2M^2$ is at most $\left(\frac{7}{8}\right)^{m_2-3 S_2T_2/\epsilon^2M^2}$.
	\end{lemma}
	 The proof of this Lemma is presented in Appendix A2. 
	 Lemma~\ref{lemma:largem_0} implies that the event that  more than $t_0$ double edges are in the output contributes at most $$\sum_{m_2>t_0} (7/8)^{m_2-3 S_2T_2/\epsilon^2M^2} (m_2-t_0)\cdot O(\Delta^2)=O(\Delta^4)
	$$ 
	to  the expected  runtime of \Gen.
	
	Let $G_0=G$, and $G_t$ denote the multigraph reached after $t$ switching steps, where $G_t=\emptyset$ if  \Gen\ terminates before completing $t$ switchings. Given $G_t\in \mathcal{H}_\bfm$, the probability that $G_{t+1}\in\mathcal{H}_{\bfm+\bfe_3}$ divided by the probability of outputting $G_t$ is 
	$$
	\frac{\UB_3(\bfm)}{\LB_3(\bfm+\bfe_3)}(1+\beta_{\bfm+\bfe_3})\leq\frac{S_3T_3}{\epsilon^3M^3}\bigg(1+\frac{4\Delta^4}{\epsilon^3M}\bigg)<1/2.
	$$
	(Here we   used the inequalities (\ref{ineq:fb}), $\max\{S_3,T_3\}<\Delta^{2}M$ and $4\Delta^4<\epsilon^3M$ in turn.) Hence the contribution to the expected runtime of \Gen\ arising from triple edges is  
	$$
	\sum_{i=1}^{\totmultcut} \left(\frac{1}{2}\right)^{i} i\cdot O(\Delta^2)=O(\Delta^2).
	$$
	
	Finally, similarly to triple edges, the probability of ever producing an edge of multiplicity at least 4 is  $O( \Delta^6/M^2)$, and hence the contribution from such multiple edges
	is $O\big((\Delta^6/M^2)\Delta^2\totmultcut\big)=O(M)$.

	Hence  \Gen\ runs with expected time $O(M)$. 
	
	The space complexity of \Gen\ is bounded by $O(mn\log (\Delta+1))$, as the main contribution comes from the adjacency matrix and each entry of the matrix is at most $\Delta$.  
	\qed \end{proof}

	\section{Uniformity, time and space complexity of \Brute}\label{sec:Brute}
	 In this section we provide a description of \Brute, prove that it is a uniform sampler of multigraphs with degree sequence $(\bfs, \bft)$, and total multiplicity at least $t_0$ and estimate its complexity. We will analyse the complexity of the algorithm while simultaneously providing the description. 
	
	 Recall that $\M(\bfs,\bft)$  denotes the set of bipartite multigraphs with degree sequence $(\bfs,\bft)$. \Brute\ will generate a uniformly random member of $\M(\bfs,\bft)$ conditional upon  the total multiplicity of multiple edges being at least $t_0$. 
	
	Before diving into the details,  we give an overall picture of how the  procedure \Brute\ works by considering a simpler problem. 
	Given a vertex $v$, imagine that we enumerate all   possibilities for the set $E(v)$   of  edges incident with $v$. Imagine also that given a particular $E(v)$, we can compute $N(E(v))$, the number of multigraphs in $\M(\bfs,\bft)$ such that the  set of edges incident  with  $v$  is precisely  $E(v)$. Then we can sample a uniformly random bipartite multigraph from $\M(\bfs,\bft)$, by first generating the edges incident with $v$ by picking $E(v)$ with probability proportional to $N(E(v))$, then  generating the set of edges incident with the second vertex in a similar way, and so on, until all edges are generated. This simple generation scheme has been proposed in the past for random generation of graphs with given degree sequence, and the problem is that there is no efficient way known to compute the numbers $N(E(v))$, which essentially requires knowing the number of graphs with a given degree sequence.
	
	The heart of  \Brute\ includes a scheme for computing the numbers $N(E(v))$ efficiently enough for our purposes. This is slightly complicated by the requirement that the  total number of edges contained in multiple edges must be greater than a given parameter $\totmultcut$.  Consequently we will consider an additional parameter $t$, being the total multiplicity of multiple edges.  We will also  use  a divide-and-conquer  scheme  to compute $N(E(v))$. This is more efficient  than from brute-force  enumeration of all possibilities, which would be too slow. It is  also more efficient  than recursive computations analogous to the generation scheme outlined above, which would require either too much time, or too much space if the required values were stored after being computed. After describing such a computation scheme for $N(E(v))$, we will show how to sample $E(v)$ with probability proportional to $N(E(v))$ within the required run time.
	
	 Recall that $\sum_{i=1}^{m} s_i=\sum_{i=1}^{n} t_i\ = M$, define the set  
$$
\mathcal{D}_1=\left\{ (g_1, \ldots, g_m, h_1, \ldots, h_n) \middle| \begin{array}{c}
	0\leq g_1\leq \cdots \leq g_m \leq \Delta, \quad 0\leq h_1\leq \cdots \leq h_n\leq \Delta,\\
	\displaystyle{\sum_{i=1}^{m} g_i=\sum_{i=1}^{n} h_i\leq M}
\end{array}\right\}
$$  
	and note that  $|\mathcal{D}_1|\leq \binom{n+\Delta}{\Delta}\binom{m+\Delta}{\Delta}$.
	Also define $\mathcal{D}$ to be  the set of all possible sequences $\textbf{d}=(\textbf{g},\textbf{h}; \totmult)$, where $(\textbf{g}, \textbf{h})\in \mathcal{D}_1$ and $0 \leq \totmult \leq M/2$.  Then $\D$ contains the set of all possible values of $(\bfg, \bfh; \totmult)$ where $(\bfg,\bfh)$ is the bi-degree sequence of a bipartite multigraph with maximum degree $\Delta$, degrees in nondecreasing order, at most $M$ edges, and  total multiplicity $\totmult$ of multiple edges. We only need to consider bi-degree sequences from $\mathcal{D}$ because permuting vertex degrees in one part of the multigraph does not affect the counts of multigraphs. 
	We have $$|\mathcal{D}| \leq \binom{n+\Delta}{\Delta} \binom{m+\Delta}{\Delta} (M/2+1) \le M(2nm)^\Delta.$$ 
	
	Given a bipartite multigraph $G$, let $\totmult(G)$ be the total multiplicity of multiple edges in $G$, and for   $\bfd =(\bfg,\bfh;\totmult)  \in \mathcal{D}$ define $\M(\bfd)=\M(\bfg,\bfh;\totmult)$ to be the set of bipartite  multigraphs $G$ with bi-degree sequence $(\bfg,\bfh)$ and $\totmult(G)=\totmult$.  Also let $N(\bfd)=|\M(\bfd)|$.

	We can now explain the divide-and-conquer approach of  calculating  the values $N(\bfd)$ recursively. It is rather simple:  for each multigraph $G\in\M(\bfd)$,  split the set $X$ into two parts, $X_1=\{1,2,\ldots, \lfloor m/2\rfloor\}$ and $X_2=[m]\setminus X_2$, and consider  the subgraphs $G_1$ and $G_2$ induced by $X_1\cup Y$  and $X_2\cup Y$ respectively. Let $(\bfg_1,\bfh_1)$ and $(\bfg_2,\bfh_2)$ denote the bi-degree sequences of $G_1$ and $G_2$ respectively, and let $\totmult_1=\totmult(G_1)$, $\totmult_2=\totmult(G_2)$. Note that $(\bfg_1,\bfg_2)$ are precisely determined by $\bfg_1=(g_1,\ldots, g_{\lfloor m/2\rfloor})$, $\bfg_2=(g_{\lfloor m/2\rfloor+1},\ldots, g_m)$,  and we also have  $\bfh_2=\bfh-\bfh_1$, $\totmult_2=\totmult-\totmult_1$.   We can   recursively compute how many possibilities there are for $G_1$, and for $G_2$, given the   parameters $\bfh_1$ and $t_1$, and moreover, such pairs of  graphs $(G_1,G_2)$ are in bijective correspondence with the possibilities for $G\in\M(\bfd)$.  Thus   	
	\bel{split}
	N(\bfd)=\sum_{\bfh_1, \totmult_1 }N(\bfg_1,\bfh_1;\totmult_1)N(\bfg_2,\bfh -\bfh_1;\totmult-\totmult_1).
	\ee
	
	Due to the large number of terms in the summation, this expression would be too slow for our purposes to calculate directly. Instead, we keep track of degree counts rather than degree sequences. The idea is to classify each possible $G_1$ according to the number $n_{ij}$ of vertices of degree $j$ in $G_1$ that have degree $i$ in $G$. For $i\in\{0,\ldots,\Delta\}$ let $n_i =|\{k\in[n]: h_k=i\}| $ be the frequency of $i$ in $\bfh$.
	Let $ {\bfp}=(n_{ij})_{0\le i,j\le \Delta}\in {\mathbb Z}_{\ge 0}^{(\Delta+1)^2}$ be such that $\sum_{j} n_{ij}=n_i$ for every $0\le i\le \Delta$. We call $\bfp$ a {\em multipartition matrix}. We say $\bfh_1$ {\em admits} the multipartition matrix $\bfp$ if  for all $i, j\in [0, \Delta]$ 
	\[
	|\{k\in[n]: h_{k}=i\ \mbox{and}\ ( \bfh_1)_k=j\}|=n_{ij}.
	\]
	It is easy to see that there are   
	$$ \prod_{ i=0}^{\Delta}\binom{n_i}{n_{i0},\ldots, n_{i\Delta}} =\prod^{\Delta}_{i=0}{n_i}!/\prod^{\Delta}_{j=0} n_{ij}!$$
	choices for $\bfh_1$ which admit the multipartition matrix $\bfp$. Moreover, for each such $\bfh_1$, the number of components equal to $j$ is determined as  $\sum_i n_{ij}$, and the number of components in $\bfh - \bfh_1$ equal to $j$ is determined as  $\sum_i n_{i(i-j)}$. Hence, each such $\bfh_1$ gives the same contribution to~\eqn{split}.  It thus suffices to consider a canonical representative $\bfh_1$ for each multipartition matrix $\bfp$, which we call $\bfh_1(\bfp)$, for instance the one in which for each $i$, the components whose indices are in $\{k: h_k=i\}$ are non-decreasing. Therefore, we may replace~\eqn{split} by
	$$
	N(\bfd)=\sum_{  \bfp,\totmult_1}N(\bfg_1,\bfh_1(\bfp),\totmult_1)N(\bfg_2,\bfh-\bfh_1(\bfp),\totmult-\totmult_1)\prod^{\Delta}_{ i=0}{n_i}!/\prod^{\Delta}_{j=0} n_{ij}!.
	$$
	This equation is used to recursively calculate $N(\bfd)$, with no storage of intermediate results. When $\bfd$ has $m=1$, the result is trivially computed as 0 or 1.
	\begin{lemma}\lab{l:Ncomplexity}
		For any $\bfd\in \mathcal{D}$, the value of $N(\bfd)$ can be computed with time complexity $O(\gamma)$ where
		$ 
		\gamma=  e^{ 3(\Delta+1)^2\log n \log m},
		$ 
		and using space complexity $O(M\log n)$.
	\end{lemma}
	
	\begin{proof} Simple arithmetic computations like computing $n_i$, $\prod^{\Delta}_{ i=0 }{n_i}!/\prod^{\Delta}_{ j=0 } n_{ij}!$, 
		addition and multiplication, take $O(M^2)$ time. A similar bound applies easily to the average time required to find the next  (in   lexicographic order)  multipartition matrix $\bfp$.  As $\bfp\in {\mathbb Z}_{\ge 0}^{ (\Delta+1)^2 }$ and each entry of $\bfp$ is at most $n$, there are at most $ n^{(\Delta+1)^2 }$ possible values for $\bfp$, and there are at most $t\le M$ choices of $t_1$. Therefore each step of recursion branches into at most $n^{ (\Delta+1)^2} M$ new steps. As the depth of the recursion is at most $1+\log_2m$,  the total time complexity of calculating $N(\bfd)$ is  
				$$
		O((n^{(\Delta+1)^2}M )^{1+(\log m)/(\log 2) }M^2).$$
 Recall that we have assumed that $5\Delta^4<M$ and $\Delta\geq 2$. It is also immediate that $M\leq \Delta \min\{m,n\}$. From this we can derive $\min\{m,n\}\geq 5\Delta^3\geq 40$ and $M\leq \min\{m^{\frac{4}{3}},n^{\frac{4}{3}}\}$. Finally, using $\log 2>2/3$, $(\Delta+1)^2\geq 9$, and $1/\log m <1/\log 40<1/3$, the above bound on the  total time complexity of calculating $N(\bfd)$ is  
 $O(e^{(\frac{11}{6}(\Delta+1)^2+10/3)\log n \log m}) = O(\gamma)$.

		For the space complexity, it requires $O((m+n)\log (\Delta+1))=O(M\log n)$  space to store $\bfd$ as each entry in $(\bfg,\bfh)$ is at most $\Delta$ and $t\le M \leq \Delta n$.  It requires $ O(\Delta^2\log n)=O(M\log n)$ space to store $\bfp$ as each entry in $\bfp$ is at most $n$. We can easily bound $N(\bfg,\bfh; t)$ by $n^{ M}$. Performing arithmetic computations such as addition and multiplication over numbers of size at most $n^{ M }$ takes $O( M  \log n)$ space. Thus the total space complexity for computing $N(\bfg,\bfh;t)$ is bounded by  $O(M\log n)$.  \qed
	\end{proof}
	\medskip
	
	We now define how \Brute, with input $(\bfs,\bft; \totmultcut)$,   samples a  member of $\bigcup_{\totmult\ge \totmultcut} \M(\bfs,\bft;\totmult)$, i.e.\ a bipartite multigraph with  bi-degree sequence
	$(\bfs ,\bft)$ and at least $\totmultcut$ edges contained in multiple edges, uniformly at random. 
	
	First, \Brute\  calculates the quantity 
	$$
	R:= \sum_{\textbf{m}\in S^+_{\totmultcut}}|\mathcal{H}_{\textbf{m}} | = \sum_{\totmult'\ge \totmultcut} N(\bfs ,\bft ;\totmult'),
	$$
	(without bothering to store the evaluations of the function $N$) and  
	chooses   a random $\totmult\ge \totmultcut$ with probability proportional to $N(\bfs ,\bft ;\totmult) $, that is, with probability
	$$
	\frac{N(\bfs ,\bft ;\totmult) }{R }.
	$$
	
	The main part of \Brute\ is then begun. It uses a (recursive) subprocedure, \sbrute,   with input parameters $(\bfs ,\bft ;\totmult)$, to sample a random member of $\M(\bfs,\bft;\totmult)$. This begins with the vertex sets $X$ and $Y$, but no edges.
	For $x\in[\Delta]$ define the set  
	$$
	\mathcal{A}(x)=\{(a_1, \ldots, a_n)\; | \; a_i\in\{0,1,\ldots, \Delta\}, \; \sum_{i=1}^{n}a_i=x\}.
	$$
	First, if $X\ne \emptyset$, \sbrute$(\bfs ,\bft ;\totmult)$ chooses a  vertex $v\in X$. (If $X$ is empty, which is the base case, it returns the graph with empty edge set.) Let $x=d(v)$, the degree of $v$. Then the set $\mathcal{A}(x)$ corresponds to all the  possible placements of edges incident with  $v$, where $a_i$ stands for the multiplicity of the edge  between $v$ and $i\in Y$. As every $\bfa\in \A(x)$ has at most $\Delta$ non-zero entries, and each non-zero entry is at most $\Delta$, it follows that $|\A(x)|\le \binom{n}{\Delta}\Delta^\Delta$. Let $\tau(\bfa)=\sum_{i:a_i\ge 2}a_i$, which corresponds to the total multiplicity of the edges incident with $v$, and let $\bfs_v$ denote the  sequence of length $n-1$ obtained from $\bfs$   by removing $v$'s entry $x$.

	Next, \sbrute\   chooses any $v\in X$ and chooses a random sequence $\bfa\in \mathcal{A}(d(v))$ with probability proportional to $N(\bfs_v,\bft-\bfa;\totmult-\tau(\bfa))$, that is, with probability
	$$
	\frac{N(\bfs_v,\bft-\bfa;\totmult-\tau(\bfa))}{ N(\bfs ,\bft ;\totmult )}.
	$$
	It  inserts the edges incident with $v$  according to $\bfa$.
	It  then finishes the generation of the graph by  generating a uniformly random graph on the vertex set $(X\setminus\{v\}, Y)$, by recursively calling \sbrute$(\bfs_v,\bft-\bfa;\totmult-\tau(\bfa))$. It is easy to see that by induction this results in the generation of a member $G$ of $\M(\bfs,\bft;\totmult)$ uniformly at random, i.e.\ with probability $ N(\bfs ,\bft ;\totmult )^{-1}$.
	
	After \sbrute\ has finished its job, \Brute\ reasserts control and
	with probability  
	$$
	\frac{R}{ N(\bfs ,\bft ;0 )(1+\beta_{\bf0})}\left(\frac{1-\rho}{\rho}\right)
	$$
	it outputs $G$, and otherwise restarts \MatrixGen. Here $\beta_{\bf0}$ was defined in Section~\ref{sec:parameters}.
	The above probability is well defined because
	$$
	\frac{R}{ N(\bfs ,\bft ;0 )(1+\beta_{\bf{0}})}\left(\frac{1-\rho}{\rho}\right)
	= \frac{R}{|\mathcal{H}_{\textbf{0}}|B} 
	= \frac{|\bigcup_{\textbf{m}\in S^+_{\totmultcut}}\mathcal{H}_{\textbf{m}}|}{|\mathcal{H}_{\textbf{0}}|B}\le 1 
	$$
	by Lemma~\ref{lemma:tail}.

	\begin{thm}\lab{thm:Brute}
		 Each call of \Brute\ generates each member of $\bigcup_{\totmult\ge \totmultcut} \M(\bfs,\bft;\totmult)$
		with probability $1/B|\mathcal{H}_{\textbf{0}}|$, and has time complexity at most $ M^{O(\Delta^2 \log M)}$ and space complexity $O(M\log n)$.  
	\end{thm} 
	\begin{proof}
		\Brute\ selects the number $t$ with probability $N(\bfs ,\bft ;\totmult) /R$ and then \sbrute\ generates each member of   $\M(\bfs,\bft;\totmult)$ with probability $1/N(\bfs ,\bft ;\totmult)$, which is then accepted with probability $R/B|\mathcal{H}_{\textbf{0}}|$. The product of these, $1/B|\mathcal{H}_{\textbf{0}}|$, is the probability that any given member of $\bigcup_{\totmult\ge \totmultcut} \M(\bfs,\bft;\totmult)$ is generated in a given call of \Brute.
		
		Turning to the time complexity, \Brute\ first computes $R$ in time $O(M\gamma)$ considering the bound for  $N(\bfs ,\bft ;\totmult')$ given in Lemma~\ref{l:Ncomplexity}. Then each call of the recursive procedure \sbrute\ needs to evaluate $N(\bfs_v,\bft-\bfa;\totmult-\tau(\bfa))$ for each $\bfa\in \mathcal{A}(d(v)) $. As we observed earlier, $|\A(v)|\le \binom{n}{\Delta}\Delta^\Delta\le (ne )^\Delta$, and so these evaluations require time $O(\gamma (ne )^\Delta)$.  There are $n-1$ calls of the recursion, so the overall time required by \sbrute\ is $O(n\gamma (ne )^\Delta)$, which subsumes the time required to compute $R$ and also $N(\bfs ,\bft ;0)$ in the last step of \Brute. Since $M\ge n$, this is at most $M^{O(\Delta^2 \log M)}$, as required.
		
		Lastly, we consider space complexity. By Lemma~\ref{l:Ncomplexity}, computing  $N(\bfd)$ requires $O(M\log n)$ space, and as seen in the proof of that lemma, this is sufficient to store numbers of this size. It is easy to check that there are no other significant space requirements. \qed

	\end{proof}
	
	
	\section{Proof of Theorem~\ref{thm:matrix}}\lab{sec:proofmain}
	 We first show that \MatrixGen\ generates a uniformly random multigraph from $\mathcal{M}(\bfs,\bft)$. Indeed, for a multigraph $G^\prime$ that has total multiplicity of multiple edges less than $t_0$, probability that $G^\prime$ is an output of \MatrixGen\ is $\frac{1-\rho}{(1+\beta_0)|\mathcal{H}_0|}$ by Theorem~\ref{thm:uniformgen}. Similarly, for $G^\prime$ that has total multiplicity of multiple edges at least $t_0$, the probability that $G^\prime$ is an output of \MatrixGen\ is $\frac{\rho}{B|\mathcal{H}_0|}$ by Theorem~\ref{thm:Brute}. It remains to notice that $\frac{1-\rho}{1+\beta_0}=\frac{\rho}{B}$ by definition of $\rho$ in Section~\ref{sec:parameters}, so \MatrixGen\ is a uniform sampler.   
	
	For the upper bound on the runtime of \MatrixGen, we bounded the  time required for generating a uniformly random simple bipartite graph, the number of switching steps of \Gen, the  time required in each switching step, and the contribution from \Brute. When estimating time complexity we assume that it takes $O(1)$ for arithmetic operations in \Gen,  however when estimating space complexity in \Brute, we potentially deal with large numbers and take into the account the space required to store those numbers. The  time taken to  compute with such large numbers does not affect the expected runtime estimates because there is  such  a low probability of calling \Brute.  
	
	 To complete the analysis, the following lemma shows that the probability of rejection happening during a single run of \Gen\ is bounded away from zero. Thus,  \Gen\  restarts a constant number of times in expectation. 
	\begin{lemma}\label{lemma:rej}
			 For some constant $c>0$, when   $M$ is sufficiently large the probability that none of  f-, b-, or $\beta$-rejection happens during a single run of \Gen\ is at least $c$. 
	\end{lemma}
		The proof of this lemma is quite cumbersome and is postponed to Appendix A2. 
		
		In view of  Lemma~\ref{lemma:rej}, we only need to estimate the runtime of (a single instance of) \Gen, which is $O(M)$ by Theorem~\ref{thm:GenTimeSpace}. Hence, \Gen\ contributes at most $O(M)$ to the time complexity of \MatrixGen.

		\Brute\  runs in superpolynomial time if ever called. However,  the probability $\rho$ that  \Brute\ is ever called is bounded by $B$ which, according to Remark~\ref{remark:B}, is at most $M^{-\Omega(M)}$. The runtime of \Brute\ is at most $M^{O(\Delta^2 \log M)}$, as shown in Theorem~\ref{thm:Brute}, so the contribution of \Brute\ to the expected runtime of \MatrixGen\ is $o(1)$. Thus the expected runtime for \MatrixGen\ is $O(M)$.
	
		As for the space complexity, from Theorem~\ref{thm:GenTimeSpace} the space complexity of \Gen\ is $O(mn\log \Delta)$. For \Brute, as proved in Theorem~\ref{thm:Brute}, the space complexity is at most $O(M \log n)$. 
	Hence the space complexity of \MatrixGen\ is $O(mn\log \Delta)$.
	
	\section{Algorithm \GraphGen}
	\lab{sec:multigraph}  
	
	The algorithm \MatrixGen\ can be modified for generation of multigraphs with given degrees $\bfs=(s_1,\ldots,s_n)$. 
	Let $\Delta$ denote the maximum degree of $\bfs$, and define
	\begin{align*}
	M&=\sum_{i\in [n]} s_i; \quad \quad M_k=\sum_{i\in [n]} (s_i)_k \quad \text{for all $k\ge 2$.}
	\end{align*}
	
	We modify the definition of $t$-switching by no longer requiring its first condition, which was the one ensuring that the chosen vertices came from appropriate sides of the bipartition. The parameters $\totmultcut, \rho, \UB_k(\cdot), \LB_k(\cdot), \LB_k(\cdot, i)$ are redefined below. 
	
	\GraphGen\ first obtains a random simple graph with degree sequence $\bfs$ by calling the algorithm \SimpleGen\ from~\cite{arman19}. \SimpleGen\ is a linear-time algorithm which generates a uniformly random simple graph with degree sequence $\bfs$ when $\Delta=O(M^{1/4})$. After that, \GraphGen\ calls \Brute\ with probability $\rho$ and calls \Gen\ with probability $1-\rho$. No modifications of \Gen\ are needed except that the switchings do not need to respect to vertex bipartition, and the set of parameters in~(\ref{parameters}) require different specifications, which we give below. For \Brute,   straightforward changes have to be made in order to generate multigraphs instead of bipartite multigraphs. For instance, the changes affect the definition of set $\mathcal{D}$ and computation procedure for $N(\bfd)$.
	
	We set the parameters~\eqn{parameters} for \GraphGen\ as follows.
	
	If  possible,  choose the integer $\totmultcut>7$ such that for $\eps=1-(2\totmultcut+6\Delta^2)/M$ we have
	$$\eps^{3}> 2\Delta^4/M \qquad \mbox{and}\qquad 8\totmultcut(\totmultcut -\Delta^2-\Delta^3)\geq   (2\totmultcut+6\Delta^2)^2,$$ 
	and otherwise set $\totmultcut=0$. 
	
	For $k\geq 3$ and  $1\leq i\leq k$ set 
	$$ \underline{b}_{k}(\bfm,i)=\epsilon M, \qquad 	\underline{b}_{k}(\bfm,0)=2m_k, $$ 
	and for $1\leq i\leq 2$ set 
	\begin{align*}
	\underline{b}_{2}(\bfm,i)&=M\left(1-\frac{2S(\bfm)+4i\Delta+2\Delta^2}{M}\right),\quad 
	\underline{b}_{2}(\bfm,0)=2m_2.
	\end{align*} 
	As before 
	\begin{align*}
	\underline{b}_k(\bfm)&=\prod_{i=0}^{k}\underline{b}_k(\bfm,i)=2m_k\prod_{i=1}^{k}\underline{b}_k(\bfm,i),\quad
	\overline{f}_{k}(\bfm)=M_k^2.
	\end{align*}
	
	The paramaters $\beta_m$ are set to be $-1$ for $\bfm \not \in S_{\totmultcut}^-$. For $\bfm$ with $\ell(\bfm)=\ell\geq 3$ set
	$$\beta_{\bfm}=\frac{2\Delta^{2\ell-2}}{M^{\ell-2}\epsilon^{\ell}},$$
	and as before for $\bfm$ with $\ell(\bfm)=2$ set
	$$\beta_{\textbf{m}}=\sum_{i=2}^\Delta\frac{\overline{f}_i(\bfm)}{\underline{b}_i(\bfm+\bfe_i)}(1+\beta_{\bfm+\bfe_i}).$$
	
	For the parameter $\rho$ we use the same specification as in Section~\ref{sec:parameters}, with the exception that $S_2=T_2=M_2:$ 
	$$
	\rho=\frac{B \frac{1}{1+\beta_{\textbf{0}}}}{1+B\frac{1}{1+\beta_{\textbf{0}}}}, \quad\mbox{where~} B=4\left(\frac{3}{2(1-\epsilon)^2}+\frac{3}{4(1-\epsilon)^4}\right)^{\epsilon M\slash 2} \left(\frac{\Delta^2e}{2\epsilon^2(1-\epsilon)(\totmultcut-7)}\right)^{\totmultcut-7}.
	$$
	
	\section*{Appendix}\label{Appendix}

	\subsection*{A1.~Proofs of Lemmas~\ref{lem:probability}--\ref{lemma:tail}}\label{sec:proofoflemmas}
	
	{\em Proof of Lemma~\ref{lem:probability}. }
	\smallskip
	
	Given $\bfm \in S^-_{\totmultcut}$ it is convenient to define
	$$\beta^+_{\textbf{m}}=\sum_{i=2}^{\Delta}\frac{\overline{f}_i(\bfm)}{\underline{b}_i(\bfm+\bfe_i)}(1+\beta_{\textbf{m}+\textbf{e}_{i}}).$$
	Let $\ell(\textbf{m})=\ell$. If $\ell=2$, then statement follows from the definition of $\beta_{\textbf{m}}$ and in this case $\beta_\bfm=\beta_\bfm^+$.
	
	Assume $\ell \geq3$, then $\beta_{\textbf{m}}=\frac{4\Delta^{2\ell-2}}{\epsilon^\ell M^{\ell-2}}$, using $\ell \leq \Delta$ we get  
	\begin{align*}
	\beta^+_{\textbf{m}}&\leq \sum_{i=\ell}^\Delta \frac{S_iT_i}{(m_s+1)M^i\epsilon^i}\left(1+\frac{4\Delta^{2i-2}}{\epsilon^i M^{i-2}}\right)\\
	&\leq \sum_{i=\ell}^\Delta \frac{\Delta^{2i-2}}{\epsilon^i (m_s+1)M^{i-2}}\left(1+\frac{4\Delta^{2i-2}}{\epsilon^i M^{i-2}}\right)\\
	&\leq \frac{\Delta^{2\ell-2}}{\epsilon^\ell M^{\ell-2}}\left(1+4\frac{\Delta^{2\ell-2}}{\epsilon^\ell M^{\ell-2}}\right)+\sum_{i=\ell+1}^\Delta \frac{\Delta^{2i-2}}{\epsilon^i M^{i-2}}\left(1+\frac{4\Delta^{2i-2}}{\epsilon^i M^{i-2}}\right)\\
	&\leq 3\frac{\Delta^{2\ell-2}}{\epsilon^\ell M^{\ell-2}} \leq \beta_{\textbf{m}}.\qed
	\end{align*}
	
	\noindent{\em Proof of Lemma~\ref{lem:sub-b-bound}.}
	\smallskip
	
	We start with showing the following property of valid $i$-subsets. 
	\begin{claim}\label{claim:validsubset}
		Let $i\in[k+1]$ and $V_{i}=(u_1,v_1,\ldots,u_i,v_i)$ be an ordered subset of vertices of $G$ such that 
		\begin{itemize}
			\item $u_1\in X$, $v_1\in Y$, $u_1v_1$ is an edge of multiplicity $k$; 
			\item $u_j\in Y$, $v_j\in X$ and $u_jv_j$ is a single edge for all $j\in[2,i]$;
			\item there are no edges between $u_1$ and $u_j$, nor between $v_1$ and $v_j$ for $j\in[2,i]$.
		\end{itemize} Then $V_i$ is a valid $i$-subset of $G$ with respect to $k$-switchings. 
	\end{claim}
	\begin{proof}
		The proof is by induction on $i$. The base case $i=k+1$ is trivial. Assuming we proved the statement for $i+1$, consider the set $A=A(V_i)$ of simple edges $u_{i+1}v_{i+1}$ that are vertex disjoint from $V_i$, such that $u_1u_{i+1}$ and $v_1v_{i+1}$ are non-edges (we also assume $u_{i+1}\in Y$ and $v_{i+1}\in X$). Then $$|A|\geq M-\totmultcut-2(i+1)\Delta-2\Delta^2,$$ since there are at least $M-\totmultcut$ simple edges and at most $2i\Delta+2\Delta^2$ of those have one of its endpoints in $V_i$, or adjacent to one of $u_1$ or $v_1$. Hence $|A|\geq 1$ and for any $u_{i+1}v_{i+1}\in A$ the ordered set $$V_{i+1}=V_i+(u_{i+1}v_{i+1})=(u_1,v_1,\ldots,u_{i+1},v_{i+1})$$ satisfies the assumption of the claim. Hence $V_{i+1}$ is a valid $i+1$-subset and consequently $V_i$ is a valid $i$-subset in $G$. \hfill \qed
	\end{proof}
	Now let $S$ be a $k$-switching that produces $G$ and for some $i\in[k]$ let $V_i=V_i(S)=(u_1, v_1\ldots, u_{i}, v_{i})$. Recall that $b_k(G,V_i)$ is the number of ordered edges that are switching compatible with $V_i$ in $G$. Claim~\ref{claim:validsubset} implies that $b_k(G,V_i)$ is the number of simple ordered edges $uv$ that are vertex disjoint from $V_i$ and such that $u_1u$ and $v_1v$ are non-edges. Therefore $b_k(G,V_i)\leq M$. On other hand, invalid choices of edges constitute of choosing a multiple edge (at most $S(\bfm)$ ways to do this), choosing an edge with endpoint in $V_i$ (at most $2i(\Delta-1)$ways to do this), or choosing an edge with one of its endpoints adjacent to $u_1$ or $v_1$ (at most $2(\Delta-2)(\Delta-1)$ ways). Hence
	$$M\left(1-\frac{S(\bfm)-2i\Delta-2\Delta^2}{M}\right)\leq b_{k}(G,V_i)\leq M.$$
	Recall that $S(\bfm)\leq \totmultcut$ and $\totmultcut=(1-\epsilon)M+4\Delta^2$, so
	$$\epsilon M\leq b_{k}(G,V_i)\leq M.$$
	
	According to Claim~\ref{claim:validsubset}, $b_k(G,V_0(S))$ is the number of possible choices of a multiple edge of multiplicity $k$ in $G$ and is equal to $m_k$, which is also equal to $\LB_k(\bfm, 0)$. \hfill \qed

\bigskip
	\noindent {\em Proof of Lemma~\ref{lemma:tbounds}.}
	\smallskip

	To perform a valid $k$-switching on $G$ we need to choose two $k$-stars $u_1,u_2\ldots,u_{k+1}$ and $v_1,v_2\ldots, v_{k+1}$, where $u_1\in X$ and $v_1\in Y$ and $u_i\in Y$, $v_i\in X$ for all $i\in[2,k+1]$. This can be done in at most $S_kT_k$ ways, hence an upper bound on $f_k(G)$. As for $b_{k}(G,S)$, the inequalities follow from Lemma~\ref{lem:sub-b-bound} and recalling that
	$$b_k(G,S)=\prod_{i=0}^{k} b_k(G, V_i(S)),  \quad \quad \LB_k(\bfm)=\prod_{i=0}^{k} \LB_k(\bfm, i).$$
	
	Finally, we show the lower bound for $f_2(G)$. There are at most $S_2$ ways to choose a labeled path $u_2u_1u_3$ and at most $T_2$ ways to choose a labelled path $v_2v_1v_3$, hence, there are at most $S_2T_2$ ways to choose two labeled paths. For the lower bound, we need to subtract the following choices: some of the vertices in path $u_2u_1u_3$ coincide with some in $v_2v_1v_3$ (there are at most $2\Delta^2(T_2+S_2)$ choices when this happens); at least one of the edges $u_1v_1$, $u_2v_2$ or $u_3v_3$ are present in $G$ (at most $\frac{1}{2}(T_2+S_2)3\Delta^3$ choices); or some of the edges form a multiple edge (at most $2S(\bfm)(T_2+S_2)\Delta$ choices). \hfill \qed
\bigskip

\noindent{\em Proof of Lemma~\ref{lem:beta}.~} 
	\smallskip
	
	Before proving the Lemma~\ref{lem:beta} we establish some inequalities for the size of sets $\mathcal{H}_{\bfm}$.
	First, let $b_{k}(G')$ be the number of $k$-switchings that produce a multigraph $G'\in \mathcal{H}_{\bfm+\bfe_k}$. As every $k$-switching $S$ that produces $G'$ can be identified with a valid $k+1$-subset $V_{k+1}(S)$ of $G'$, $b_{k}(G')$ is equal to the number of the valid $k+1$-subsets of $G'$. According to Lemma~\ref{lemma:tbounds}, there are at least $b_{k}(\bfm+\bfe_k,0)$ valid $1$-subsets in $G'$, and for $i\in[k]$ each valid $i$-subset can be extended to a valid $i+1$-subset in at least $b_{k}(\bfm,i)$ ways. Hence there are at least $\prod_{i=0}^kb_{k}(\bfm+\bfe_k,i)$ ways to choose a valid $k+1$-subset in $G'$, and consequently $b_{k}(G')\geq b_{k}(\bfm+\bfe_k)$.
	
	Therefore, for a sequence $\textbf{m}$ with $S(\textbf{m})\leq \totmultcut$ there are at least $|\mathcal{H}_{\textbf{m}+\bfe_{k}}| \underline{b}_k(\bfm+\bfe_{k})$ $k$-switchings that produce a multigraph in $\mathcal{H}_{\bfm+\bfe_k}$ from a multigraph in $\mathcal{H}_{\bfm}$. Now, it follows directly from Lemma~\ref{lemma:tbounds} that for all $k\geq 2$, 
	$$|\mathcal{H}_{\textbf{m}+\bfe_{k}}| \underline{b}_k(\bfm+\bfe_{k}) \leq |\mathcal{H}_{\textbf{m}}| \overline{f}_k(\bfm).$$ 
	Since $\underline{b}_k(\bfm)\geq m_kM^k\epsilon^k$, we have
	\begin{align}\label{ineq:H_i}
	\frac{|\mathcal{H}_{\textbf{m}+e_{k}}|}{|\mathcal{H}_{\textbf{m}}|}\leq \frac{S_kT_k}{(m_k+1)M^{k}} \cdot \frac{1}{\epsilon^k} \leq \frac{\Delta^{2k-2}}{\epsilon^k(m_k+1)M^{k-2}}
	\end{align}
	for all $k\geq 2$ and $\bfm$ with $S(\bfm)\leq \totmultcut$.

	The following definition is useful for the next Claim. For a sequence $\textbf{m}\in S_{\totmultcut}^-$ with $k=\ell(\textbf{m})$ define $U(\textbf{m})$ to be a set of all sequences $\textbf{m}^\prime=(0,m^{\prime}_2, \ldots, m^{\prime}_{\Delta})\in S_{\totmultcut}^-$ such that $\bfm \prec \bfm'$. Finally, recall that $\mathcal{H}^{+}_{\textbf{m}}=\cup_{\textbf{m}^\prime \in U(\textbf{m})}\mathcal{H}_{\textbf{m}^\prime}$. The following Claim motivates the definition of $\beta_{
		\bfm}$ for $\ell(\bfm)\geq 3$. 
	\begin{claim}\label{lemma:bound_H+}
		Assume that $4\Delta^4<M$ and $\textbf{m}\in S_{\totmultcut}^-$  with $\ell(\textbf{m})=\ell\geq 3$. Then 
		$$\frac{|\mathcal{H}^{+}_{\textbf{m}}|}{|\mathcal{H}_{\textbf{m}}|}\leq \frac{4\Delta^{2\ell-2}}{\epsilon^\ell M^{\ell-2}}.$$
	\end{claim}
	\begin{proof}
		Let $\textbf{m}\in S_{\totmultcut}^-$ and let $\ell=\ell(\textbf{m})$. By inequality (\ref{ineq:H_i}), for all $k \in [\ell,\Delta]$ we have
		$$\frac{|\mathcal{H}_{\textbf{m}+e_k}|}{|\mathcal{H}_{\textbf{m}}|}\leq \frac{\Delta^{2k-2}}{\epsilon^k(m_k+1)M^{k-2}}.$$
		
		Hence, for all integers $x\geq 0$ we have 
		$$\frac{|\mathcal{H}_{\textbf{m}+xe_k}|}{|\mathcal{H}_{\textbf{m}}|} \leq \left(\frac{\Delta^{2k-2}}{\epsilon^kM^{k-2}}\right)^x \frac{1}{(m_k+x)_x}.$$

		Now, each $\textbf{m}^{\prime}\in U(\textbf{m})$ can be considered as $\textbf{m}^\prime=\textbf{m}+x_\ell e_\ell+\ldots+x_{\Delta}e_{\Delta}$ for some non-negative $x_\ell, \ldots, x_\Delta$ and hence
		\begin{align*}
		\frac{|\mathcal{H}_{\textbf{m}^\prime}|}{|\mathcal{H}_{\textbf{m}}|}&=\frac{|\mathcal{H}_{\textbf{m}+x_\ell e_\ell}|}{|\mathcal{H}_{\textbf{m}}|}\frac{|\mathcal{H}_{\textbf{m}+x_\ell e_\ell +x_{\ell+1}e_{\ell+1}}|}{|\mathcal{H}_{\textbf{m}+x_\ell e_\ell}|}\cdots \frac{|\mathcal{H}_{\textbf{m}^\prime}|}{|\mathcal{H}_{\textbf{m}^\prime-x_\Delta e_\Delta}|}\\
		&\leq \left(\frac{\Delta^{2\ell-2}}{\epsilon^\ell M^{\ell-2}}\right)^{x_\ell} \frac{1}{(m_\ell+x_\ell)_{x_\ell}}\prod_{k=\ell+1}^{\Delta} \left(\frac{\Delta ^{2k-2}}{\epsilon^k M^{k-2}}\right)^{x_k} \frac{1}{x_k!}.
		\end{align*}
		
		So, finally we have 
		\begin{align*}
		\frac{|\mathcal{H}_{\textbf{m}}^+|}{|\mathcal{H}_{\textbf{m}}|} & \leq \left(\sum_{i=0}^{\infty}\left(\frac{\Delta^{2\ell-2}}{\epsilon^\ell M^{\ell-2}}\right)^i\frac{1}{(m_\ell+i)_i}\right)\prod_{k=\ell+1}^{\Delta}\left(\sum_{i=0}^{\infty}\left(\frac{\Delta^{2k-2}}{\epsilon^k M^{k-2}}\right)^{i} \frac{1}{i!}\right)-1
		\end{align*}
		Now, $\frac{\Delta^{2\ell-2}}{\epsilon^\ell M^{\ell-2}} \leq \frac{1}{4}$ implies 
		\begin{align*}
		\frac{|\mathcal{H}_{\textbf{m}}^+|}{|\mathcal{H}_{\textbf{m}}|} & \leq \left(1+2\left(\frac{\Delta^{2\ell-2}}{\epsilon^\ell M^{\ell-2}}\right)\right)\prod_{k=\ell+1}^{\Delta}\left(\exp\left(\frac{\Delta^{2k-2}}{\epsilon^k M^{k-2}}\right)\right)-1\\
		&\leq \left(1+2\left(\frac{\Delta^{2\ell-2}}{\epsilon^\ell M^{\ell-2}}\right)\right) \exp\left(\frac{\Delta^{2\ell-2 }}{\epsilon^\ell M^{\ell-2}}\right)-1\\
		&\leq \left(1+2\frac{\Delta^{2\ell-2}}{\epsilon^\ell M^{\ell-2}}\right)\left(1+1.2\left(\frac{\Delta^{2\ell-2}}{\epsilon^\ell M^{\ell-2}}\right)\right)-1\\
		&\leq \frac{4\Delta^{2\ell-2}}{\epsilon^\ell M^{\ell-2}}. \hspace{300pt} \qed
		\end{align*}
	\end{proof}

	Finally we are ready to prove Lemma~\ref{lem:beta}.
	
	We proceed by induction. Statement follows for all $\bfm$ with $\ell(\bfm)\geq 3$ from Claim~\ref{lemma:bound_H+}. So we may assume that $\ell(\bfm)=2$ and for all $\textbf{m}^{\prime} \in U(\textbf{m})$ we proved the Lemma.
	Then,
	\begin{align*}
	\beta_{\textbf{m}}&=\sum_{i=2}^{\Delta}\frac{\overline{f}_i(\bfm)}{\underline{b}_i(\bfm+\bfe_i)}(1+\beta_{\textbf{m}+\textbf{e}_{i}})\\
	&\geq \sum_{\scriptsize \begin{array}{c} 2 \leq i\leq \Delta \\ \textbf{m}+\textbf{e}_i
		\in U(\textbf{m})\end{array}}\frac{\overline{f}_i(\bfm)}{\underline{b}_i({\bfm+\bfe_i})}\left(1+\frac{|\mathcal{H}^+_{\textbf{m}+\textbf{e}_{i}}|}{|\mathcal{H}_{\textbf{m}+\textbf{e}_{i}}|}\right)\\
	&\geq \sum_{\scriptsize \begin{array}{c} 2 \leq i\leq \Delta \\ \textbf{m}+\textbf{e}_i
		\in U(\textbf{m})\end{array}}\frac{|\mathcal{H}_{\textbf{m}+\textbf{e}_{i}}|}{|\mathcal{H}_{\textbf{m}}|}\left(1+\frac{|\mathcal{H}^+_{\textbf{m}+\textbf{e}_{i}}|}{|\mathcal{H}_{\textbf{m}+\textbf{e}_{i}}|}\right)\\
	&\geq \sum_{\scriptsize \begin{array}{c} 2 \leq i\leq \Delta \\ \textbf{m}+\textbf{e}_i
		\in U(\textbf{m})\end{array}}\frac{|\mathcal{H}_{\textbf{m}+\textbf{e}_{i}}|+|\mathcal{H}^+_{\textbf{m}+\textbf{e}_{i}}|}{|\mathcal{H}_{\textbf{m}}|}=\frac{|\mathcal{H}^{+}_{\textbf{m}}|}{|\mathcal{H}_{\textbf{m}}|}. \qed
	\end{align*}
	\medskip
	
\noindent	{\em Proof of Lemma~\ref{lemma:tail}.}
\smallskip

	Recall that $\totmultcut=(1-\epsilon)M-4\Delta^2$. For $t\geq0$ define 
	$\mathcal{S}_t=\bigcup_{S(\bfm)=t}\mathcal{H}_{\bfm}$ and observe that  $\bigcup_{\bfm\in S^+_{\totmultcut}}\mathcal{H}_{\bfm}=\bigcup_{t\geq \totmultcut}\mathcal{S}_t$.  The proof of the lemma is based on the following two claims. We first estimate $|\mathcal{S}_t|\slash |\mathcal{H}_0|$ for $t$ close to $\totmultcut$ and then estimate size of $\bigcup_{\bfm \in S_{\totmultcut}^{+}}\mathcal{H}_{\bfm}$ via sizes of four appropriate $\mathcal{S}_t$.
	
	\begin{claim}\label{claim:ratioto0}
		For $t\leq \totmultcut$ the following inequality holds
		$$\frac{|\mathcal{S}_t|}{|\mathcal{H}_{\bf{0}}|}\leq \left(\frac{\Delta^2e}{\epsilon^2(1-\epsilon)t}\right)^t.$$
	\end{claim}
	\begin{proof}
		Iterative application of inequality (\ref{ineq:H_i}) implies that for all $\bfm$ with $S(\bfm)\leq \totmultcut$
		$$\frac{|\mathcal{H}_{\bf{m}}|}{|\mathcal{H}_{\bf{0}}|} \leq \prod_{k=2}^{\Delta} \left(\frac{\Delta ^{2k-2}}{\epsilon^k M^{k-2}}\right)^{m_k} \frac{1}{m_k!}.$$
		Hence, the coefficient of $x^t$ in the Taylor expansion of 
		$$f(x)=\prod_{k=2}^{\Delta}\sum_{i=0}^{\infty}\left(\frac{\Delta^{2k-2}}{\epsilon^k M^{k-2}}\right)^i\frac{x^i}{i!}$$ is an upper bound for $|\mathcal{S}_t|\slash |\mathcal{H}_{\bf{0}}|$. On other hand,  $f(x)=\exp\left({\sum_{k=2}^{\Delta}\left(\frac{\Delta^{2k-2}}{\epsilon^kM^{k-2}}\right)}x\right)$, so we conclude that 
		$$\frac{|\mathcal{S}_t|}{|\mathcal{H}_{\bf{0}}|}\leq \left(\sum_{k=2}^{\Delta}\frac{\Delta^{2k-2}}{\epsilon^kM^{k-2}}\right)^t\frac{1}{t!} \leq \left(\frac{\Delta^{2}}{\epsilon^2(1-\epsilon)}\right)^t\left(\frac{e}{t}\right)^t.
		\hfill\qed$$\end{proof}
	\begin{claim}\label{claim:brute16}
		For $t\geq \totmultcut+3$ we have 
		$$\frac{|\mathcal{S}_t|+|\mathcal{S}_{t-1}|+|\mathcal{S}_{t-2}|+|\mathcal{S}_{t-3}|}{|\mathcal{S}_{t-4}|+|\mathcal{S}_{t-5}|+|\mathcal{S}_{t-6}|+|\mathcal{S}_{t-7}|}\leq C_1,$$
		where $C_1=3\left(\frac{1}{2(1-\epsilon)^2}+\frac{1}{4(1-\epsilon)^4}\right)$
	\end{claim}
	\begin{proof}
		To prove this Claim we make a use of the auxiliary switching defined as follows. Assume that $u_1v_1$ and $u_2v_2$ are multiple edges in a multigraph $G$ with $u_1,u_2\in X$, such that $u_1v_2$ and $u_2v_1$ are non-edges. Auxiliary switching reduces multiplicities of $u_1v_1$ and $u_2v_2$ by $1$ and adds edges $u_1v_2$ and $u_2v_1$. Auxiliary switchings maps multigraphs from $\mathcal{S}_\ell$ to multigraphs in $\mathcal{S}_{\ell-2}\bigcup \mathcal{S}_{\ell-3}\bigcup \mathcal{S}_{\ell-4}$ depending on the multiplicities of $u_1v_1$ and $u_2v_2$. 
		For every multigraph $G\in \mathcal{S}_\ell$ there are at least $\frac{\ell}{\Delta}(\frac{\ell}{\Delta}-(\Delta-2)-(\Delta-2)(\Delta-1))$
		auxiliary switchings that can be performed on $G$. On other hand, for each $G'\in \mathcal{S}_{\ell-2}\bigcup \mathcal{S}_{\ell-3}\bigcup \mathcal{S}_{\ell-4}$ there are at most $M(\Delta-1)^2$ auxiliary switching that result in $G'$. Therefore for $\ell\geq \totmultcut$ we have a bound 
		$$\frac{|\mathcal{S}_\ell|}{|\mathcal{S}_{\ell-2}\bigcup \mathcal{S}_{\ell-3}\bigcup \mathcal{S}_{\ell-4}|}\leq \frac{M(\Delta-1)^2}{\frac{\ell}{\Delta}(\frac{\ell}{\Delta}-(\Delta-2)-(\Delta-2)(\Delta-1))}.$$
		Recall that $4\Delta^4\leq M$, $\totmultcut=(1-\epsilon)M-4\Delta^2$ and $\totmultcut(\totmultcut-\Delta^2-\Delta^3)\geq \frac{1}{2}(1-\epsilon)^2M^2$, then
		\begin{align}\label{eq:NewUB}\frac{|\mathcal{S}_\ell|}{|\mathcal{S}_{\ell-2}\bigcup \mathcal{S}_{\ell-3}\bigcup \mathcal{S}_{\ell-4}|}&\leq \frac{2\Delta^4}{(1-\epsilon)^2M}\leq \frac{1}{2(1-\epsilon)^2}. \end{align}
		Now, set $C_0=\frac{1}{2(1-\epsilon)^2}$, applying inequality (\ref{eq:NewUB}) recursively for values of $\ell\in \{t, t-1,t-2,t-3\}$ yields
		$$\frac{|\mathcal{S}_t|+|\mathcal{S}_{t-1}|+|\mathcal{S}_{t-2}|+|\mathcal{S}_{t-3}|}{|\mathcal{S}_{t-4}|+|\mathcal{S}_{t-5}|+|\mathcal{S}_{t-6}|+|\mathcal{S}_{t-7}|}\leq 3(C_0^2+C_0).\hfill \qed$$

	\end{proof}
	
	Now we are ready to prove the lemma. First, let $k=\lceil\frac{M-\totmultcut+1}{4}\rceil$ and $x=M-4k-4$, then in view of Claim~\ref{claim:brute16}, 
	$$\frac{|\bigcup_{t\geq \totmultcut}\mathcal{S}_t|}{|\mathcal{S}_x|+|\mathcal{S}_{x+1}|+|\mathcal{S}_{x+2}|+|\mathcal{S}_{x+3}|}\leq C_1+C_1^2+\ldots C_1^k\leq C_1^{k+1}.$$
	
	Now, it is easy to verify that $\totmultcut-4\geq x\geq \totmultcut-7$ and in view of Claim~\ref{claim:ratioto0}
	$$\frac{|\mathcal{S}_x|+|\mathcal{S}_{x+1}|+|\mathcal{S}_{x+2}|+|\mathcal{S}_{x+3}|}{|\mathcal{H}_0|}\leq 4 \left(\frac{\Delta^2e}{\epsilon^2(1-\epsilon)(\totmultcut-7)}\right)^{\totmultcut-7}.$$
	
	Therefore, we obtain 
	$$\frac{|\bigcup_{\bfm\in S^{+}_{\totmultcut}}\mathcal{H}_{\bfm}|}{|\mathcal{H}_0|}=\frac{|\bigcup_{t\geq \totmultcut}\mathcal{S}_t|}{|\mathcal{H}_0|}\leq 4C_1^{k+1} \left(\frac{\Delta^2e}{\epsilon^2(1-\epsilon)(\totmultcut-7)}\right)^{\totmultcut-7}.$$
	Finally, the statement of the lemma follows from observing that $k+1\leq \epsilon M \slash 2$.

	
	\subsection*{A2. Proof of Lemmas~\ref{lemma:largem_0} and~\ref{lemma:rej}.}\label{sec:prooftwolemmas}

	We first estimate the probability that algorithm \Gen\ creates a graph $G\in \mathcal{H}_\bfm$ with $\ell(\bfm)=2$ and with $m_2>3S_2T_2/\epsilon^2M^2$. 
	\bigskip
	
\noindent{\em Proof of Lemma~\ref{lemma:largem_0}.~}
\smallskip

	We first note that if $\ell(\bfm)=2$ and $m_2>3S_2T_2/\epsilon^2M^2$, then 
	$$\frac{\overline{f}_2(\bfm)}{\underline{b}_2(\bfm+\bfe_2)}\leq \frac{\epsilon^2}{3} \frac{1}{\big(1-(S(m)+4\Delta+2\Delta^2)/M \big)^2}\leq \frac{1}{3},$$
	and for $k>2$,  inequality (\ref{ineq:fb}) implies 
	$$\frac{\overline{f}_k(\bfm)}{\underline{b}_k(\bfm+\bfe_k)} \leq   \frac{\Delta^{2k-2}}{\epsilon^kM^{k-2}}.$$
	
	Once we reached a graph $G\in \mathcal{H}_{\bfm}$ with $m_2>3S_2T_2/\epsilon^2M^2$, the probability that we decide to not output $G$ and increase the number of double edges in $G$ is at most \begin{align*}
	\frac{\beta_{\bfm}}{1+\beta_{\bfm}}&\frac{\overline{f}_2(\bfm)}{\underline{b}_2(\bfm+\bfe_2)}(1+\beta_{\bfm+\bfe_2})\cdot\frac{1}{ \beta_{\bfm}}\leq \frac{\overline{f}_2(\bfm)}{\underline{b}_2(\bfm+\bfe_2)}(1+\beta_{\bfm+\bfe_2})\\
	&\leq \frac{1}{3}\left(1+\frac{\overline{f}_2(\bfm+\bfe_2)}{\underline{b}_2(\bfm+2\bfe_2)}(1+\beta_{\bfm+2\bfe_2})+\sum_{i=3}^\Delta \frac{\Delta^{2i-2}}{\epsilon^iM^{i-2}}(1+3\frac{\Delta^{2i-2}}{\epsilon^iM^{i-2}})\right)\\
	&\leq \frac{1}{3}\left(1+3\frac{\Delta^{4}}{\epsilon^3M}\right)+\frac{1}{3}\frac{\overline{f}_2(\bfm+\bfe_2)}{\underline{b}_2(\bfm+2\bfe_2)}(1+\beta_{\bfm+2\bfe_2})\\
	&\leq \left(\frac{1}{3}+\left(\frac{1}{3}\right)^2\right)\left(1+3\frac{\Delta^{4}}{\epsilon^3M}\right)+\left(\frac{1}{3}\right)^2\frac{\overline{f}_2(\bfm+2\bfe_2)}{\underline{b}_2(\bfm+3\bfe_2)}(1+\beta_{\bfm+3\bfe_2})\\
	&\leq \left(\frac{1}{3}+\left(\frac{1}{3}\right)^2+\ldots\right)(1+3\frac{\Delta^4}{\epsilon^3M})\leq \frac{7}{8}.
	\end{align*}
	
	Hence, the probability of deciding to not output $G$ and increase the number of double edges in $G$ is at most $\frac{7}{8}$.
	
	Condition on reaching a graph in $\mathcal{H}_{\bfm'}$ with $m'_2=3S_2T_2/\epsilon^2M^2$ in \Gen, probability that we reach a graph in $\mathcal{H}_{\bfm}$ with $m_2$ double edges  is at most $\left(\frac{7}{8}\right)^{m_2-3S_2T_2/\epsilon^2M^2}$. So, unconditional probability is also at most that large.  \hfill \qed
	\bigskip

\noindent{\em Proof of Lemma~\ref{lemma:rej}.~}
\ss

	We separate a single run of \Gen\ into two parts: part (a) is when the current graph $G\in \mathcal{H}_{\textbf{m}}$ with $\ell(\textbf{m})=2$ and part (b) is when $\ell(\textbf{m})>2$. We will show that probability of ever reaching part (b) is at most  $3/4$.

	For now we consider part (a). Set  $t_0=\max\{6\epsilon^2,1\}$. 
	\ss
	
	\noindent {\bf Case 1:}  $  S_2T_2/M^2 \geq t_0$.
	\ss
	
	The probability that in part (a) we ever reach a graph $G$ with more than $4 S_2T_2/\epsilon^2M^2 $ double edges, by Lemma~\ref{lemma:largem_0}, is at most $\left(\frac{7}{8}\right)^{S_2T_2/\epsilon^2M^2}<\frac{1}{2}$. Hence, the probability of rejection happening on some $G$ with more than $4S_2T_2/\epsilon^2M^2$ double edges is at most $\frac{1}{2}$. So, we need to consider only the case when part (a) runs for at most $4S_2T_2/\epsilon^2M^2$ iterations.
	
	{\textbf{ b-rejection}.}
	The probability that b-rejection does not happen during a single switching step is at least
	$$\prod_{i=1}^{2} \frac{\underline{b}_2(\bfm+\bfe_2,i)}{M}\geq \left(1-\frac{4S_2T_2/\epsilon^2M^2+4\Delta+2\Delta^2}{M}\right)^2.$$
	Hence the probability of b-rejection not happening during part (a) of a single run of \Gen\ is at least  
	$$\left(1-\frac{4S_2T_2/\epsilon^2M^2+4\Delta+2\Delta^2}{M}\right)^{8S_2T_2/\epsilon^2M^2}.$$
	The last quantity is   $\exp(-O( \Delta^2S_2T_2/M^2))=\exp(-O( \Delta^4/M))$. 
	
	{\textbf{f-rejection}.}
	Similarly, the probability of not having f-rejection during part (a) of one run of \Gen\ is at least
	$$\left(1-\frac{\Delta(S_2+T_2)(2(4S_2T_2/\epsilon^2M^2)+2\Delta+1.5\Delta^2)}{S_2T_2}\right)^{4S_2T_2/\epsilon^2M^2}.$$
	This is   $\exp(-O( \Delta^3(S_2+T_2)/M^2 ))=\exp(-O( \Delta^4/M))$.  
	
	During part (a), $\beta$-rejection does not happen because in this case 
	$$\beta_{\bfm}=\sum_{s=2}^\Delta \frac{\UB_s{(\bfm)}}{\LB_s{(\bfm+\bfe_s})}(1+\beta_{\bfm+\bfe_s}).$$
	
	Combining these conclusions, we deduce that the probability of not having any rejection during   part (a) of one run of \Gen\ is at least $\frac{1}{2}\exp(-O( \Delta^4/M ))$, which is at least $c_1$ for some $c_1>0$.

	\ss
	\noindent {\bf Case 2:}  $ S_2T_2/M^2 <t_0$.
	\ss
	
	This  is similar to Case 1. The probability of reaching a graph with more than $6+3t_0/\epsilon^2$ double edges is at most $1/2$. So it is enough to consider the  case when part (a) runs only for $6+3t_0/\epsilon^2$ iterations. In this case $\beta_{\bf{0}}>c_2$ for some absolute constant $c_2>0$, hence the probability of not having rejection during   part (a) is at least $1/2(1+c_2)$. In Case 2 we define $c_1=1/2(1+c_2)$.\medskip

	Finally, we estimate the probability of ever having part (b) during a single run of \Gen. This requires that some $G$ was generated in $\mathcal{H}_{\textbf{m}}$, where $\textbf{m}= (0,m_2,0,\ldots,0)$, and it was decided not to output $G$, and then  some $s\in[3,\Delta]$ was chosen. We say that in this case  part  (b) was initiated from $G$. Now
	\begin{align*}\frac{\mathbb{P}(\text{part (b) initiated from} \; G )}{\mathbb{P}(\text{output}\;  G )}&=\left.\frac{\beta_{\textbf{m}}}{1+\beta_{\textbf{m}}}\left(\sum_{k=3}^{\Delta}\frac{\UB_k{(\bfm)}}{\LB_k{(\bfm+\bfe_k})}(1+\beta_{\textbf{m}+\textbf{e}_{k}})\slash \beta_{\textbf{m}}\right)\middle\slash \frac{1}{1+\beta_{\textbf{m}}}\right.\\
	&=\sum_{k=3}^{\Delta}\frac{\UB_k{(\bfm)}}{\LB_k{(\bfm+\bfe_k})}(1+\frac{4\Delta^{2k-2}}{\epsilon^kM^{k-2}})\\
	&\leq \sum_{k=3}^{\Delta}\frac{\Delta^{2k-2}}{\epsilon^kM^{k-2}}(1+\frac{4\Delta^{2k-2}}{\epsilon^kM^{k-2}})\leq  3 \frac{\Delta^4}{\epsilon^3M}.
	\end{align*}
	(The very last inequality is based on the assumption that $M$ is large enough and on the inequality  $ \Delta^4/\epsilon^3M<\frac{1}{4}$.)
	
	Hence, 
	\begin{align*}\frac{\mathbb{P}(\text{part (b) initiated from some $G$)}}{\mathbb{P}(\text{some} \; G \; \text{outputted in part (a) )}}&=\frac{\sum_{G \in \mathcal{H}_{\textbf{m}}, \; \ell(\textbf{m})=2}\mathbb{P}(\text{part (b) initiated  from} \; G )}{\sum_{G \in \mathcal{H}_{\textbf{m}}, \; \ell(\textbf{m})=2}\mathbb{P}(\mbox{$G$  outputted})}\\
	&\leq  3 \frac{\Delta^4}{\epsilon^3M}\leq \frac{3}{4}.
	\end{align*}
	Therefore, the probability of ever initiating part $(b)$ is at most  $3/4$. We deduce that the probability of no rejection happening during a single run of \Gen\ is at least $c_1/4$. \qed

	\bibliographystyle{plain}
	\bibliography{matrix}

\end{document}